\RequirePackage{fix-cm}
\documentclass[]{elsarticle}

\makeatletter
\def\ps@pprintTitle{%
 \let\@oddhead\@empty
 \let\@evenhead\@empty
 \def\@oddfoot{}%
 \let\@evenfoot\@oddfoot}
\makeatother

\usepackage[margin=1in]{geometry}
\usepackage{mathrsfs}
\usepackage{amssymb, amsthm, amsmath}
\usepackage{xfrac}
\usepackage{bm}
\usepackage{enumitem}
\usepackage[colorlinks]{hyperref}

\newtheorem{proposition}{Proposition}

\newtheorem{lemma}{Lemma}

\newtheorem{corollary}{Corollary}

\newdefinition{definition}{Definition}

\newdefinition{remark}{Remark}

\newdefinition{example}{Example}

\DeclareFontFamily{U}{rsfs}{\skewchar\font127}
\DeclareFontShape{U}{rsfs}{m}{n}{<-6> rsfs5 <6-8> rsfs7 <8-> rsfs10}{}

\newcommand{\numberthis}{\addtocounter{equation}{1}\tag{\theequation}}
\DeclareMathOperator*{\esssup}{ess\,sup}
\DeclareMathOperator*{\plim}{\mathbb{P}-lim}

\begin{document}

\begin{frontmatter}

	\title{A Stochastic Calculus for Rosenblatt Processes}

	\author[KU,MFF]{Petr \v{C}oupek}
	\ead{coupek@karlin.mff.cuni.cz}

	\author[KU]{Tyrone E. Duncan\corref{cor}}
	\ead{duncan@ku.edu}

	\author[KU]{Bozenna Pasik-Duncan}
	\ead{bozenna@ku.edu}

	\cortext[cor]{Corresponding author}
	\address[KU]{University of Kansas, Department of Mathematics, 1460 Jayhawk Blvd., Lawrence, 660 45, Kansas, USA}
	\address[MFF]{Charles University, Faculty of Mathematics and Physics, Sokolovsk\'a 83, Prague 8, 186 75, Czech Republic}


	\begin{abstract}
		A stochastic calculus is given for processes described by stochastic integrals with respect to fractional Brownian motions and Rosenblatt processes somewhat analogous to the stochastic calculus for It\^o processes. These processes for this stochastic calculus arise naturally from a stochastic chain rule for functionals of Rosenblatt processes; and some It\^o-type expressions are given here. Furthermore, there is some analysis of these results for their applications to problems using Rosenblatt noise.
	\end{abstract}

\end{frontmatter}

\medskip

\begin{center}
\textit{This paper is dedicated to the memory of Larry Shepp.}
\end{center}

\section{Introduction}

Self-similar stochastic processes, that are processes whose distributions are invariant under suitable scalings, can be used as mathematical models of various physical phenomena. These processes have been used for modeling in hydrology, biophysics, geophysics, telecommunication, turbulence, cognition, and finance. Typically, these self-similar processes exhibit long-range dependence, that is, their autocorrelations decay slower than exponentially. Some bibliographical guides are given by Taqqu \cite{Taqqu86}; and Willinger, Taqqu, and Erramili \cite{WilTaqErr96}; that provide applications of self-similar stochastic processes and many references.

The family of fractional Brownian motions is among the most studied self-similar stochastic processes. Fractional Brownian motion indexed by the Hurst parameter $0<H<1$, that is denoted by $B^H$ here, is a centered Gaussian stochastic process whose covariance function is given by
	\begin{equation*}
		\mathbb{E} B^H_sB_t^H = \frac{1}{2}\left(|s|^{2H}+|t|^{2H}-|s-t|^{2H}\right), \quad s,t\in\mathbb{R}.
	\end{equation*}
There are at least two reasons why fractional Brownian motions are of interest. First, these processes are self-similar, have stationary increments, and exhibit long-range dependence for $\sfrac{1}{2}<H<1$. These properties make them very attractive for practical modeling and applications. The second reason is the fact that they are Gaussian processes which makes some mathematical models using fractional noise feasible for analysis. In fact, stochastic calculus for fractional Brownian motions is fairly developed, e.g. \cite{AlosMazNua01,AlosNua03,BiaHuOksZha08, DecrUstu99,DunJakDun06,DunHuDun00}.

However, non-Gaussian data with fractal features have also been observed empirically,  e.g. \cite{Dom15} where control error in single-input-single-output (SISO) loops is analyzed. Doma\'nski has shown from data of some physical systems that the Gaussian assumption is not always appropriate. In such cases, it does not seem reasonable to use a Gaussian process such as a fractional Brownian motion as a model for these physical phenomena and the use of a Rosenblatt process can provide a useful alternative.

A Rosenblatt process with the Hurst parameter $\sfrac{1}{2}<H<1$, denoted here as $R^H$, can arise as a non-Gaussian limit of suitably normalized sums of long-range dependent random variables in a non-central limit theorem, see e.g. \cite{DobMaj79, Ros61, Taqq79}. This process admits a version with H\"older continuous sample paths (up to order $H$), has stationary increments, and is $H$-self-similar with long-range dependence. Moreover, its covariance function is the same as that of the fractional Brownian motion $B^H$. However, unlike the family of fractional Brownian motions, the family of Rosenblatt processes is not Gaussian.

A detailed history, construction, and many properties of Rosenblatt processes are given in the survey article of Taqqu \cite{Taqqu11}. Some stochastic analysis of Rosenblatt pocesses is given by Tudor in \cite{Tud08} and some properties of Rosenblatt processes are given in \cite{AbrPip06,Alb98,Pip04}. Furthemore, stochastic (partial) differential equations with additive Rosenblatt noise have also been studied, e.g. \cite{BonTud11, Cou18, CouMas17, CouMasOnd18}.

However, despite the considerable attention that Rosenblatt processes have received, there is only a limited development of a stochastic calculus and especially It\^o-type formulas for these processes.

In the pioneering work of Tudor \cite{Tud08}, a representation of Rosenblatt processes on a finite time interval is given and used to construct both Wiener-type and stochastic integrals in which Rosenblatt processes appear as the integrators. Furthemore, an It\^o-type formula for functionals of a Rosenblatt process is given for some general conditions. However, these conditions seem to be difficult to verify in specific cases and in fact in \cite{Tud08}, the conditions are only verified for the square and cube of a Rosenblatt process.

In \cite{Arr15}, a stochastic calculus with respect to Rosenblatt processes is developed by means of white-noise theory \cite{HidaKuoPottStr94}. In this framework, an It\^o-type formula for functionals of a Rosenblatt process is proved. However, this formula is given as an infinite series that involves derivatives of all orders and white-noise integrals with respect to stochastic processes obtained from the Rosenblatt process.

An important contribution is made by Arras in \cite{Arr16} that improves the results in \cite{Tud08} and provides an It\^o-type formula for functionals of Rosenblatt processes by means of Malliavin calculus on the white-noise probability space which allows to use techniques from white-noise distribution theory. This It\^o-type formula is valid for infinitely differentiable functionals with at most polynomial growth.

In the approach used here, some methods of \cite{Arr16} are used without relying on the white-noise setting. Not only functionals of Rosenblatt processes but functionals of stochastic integrals with respect to them are considered. A main results of this paper is an It\^o-type formula for $\mathscr{C}^3$ functionals with at most polynomial growth of the stochastic processes with second-order fractional differential of the form
	\begin{equation}
	\label{eq:x_t_intro}
		x_t = x_0+ \int_0^t\vartheta_s\,\mathrm{d}{s} + 2c_{H}^{B,R}\int_0^t\varphi_s\delta B_s^{\frac{H}{2}+\frac{1}{2}} + \int_0^t\psi_s\delta R_s^H
	\end{equation}
that have H\"older continuous sample paths of an order greater than $\sfrac{1}{2}$. Here, $c_H^{B,R}$ is a suitable normalizing constant. The integrals are defined using (multiple) Skorokhod integrals with respect to a Wiener process and suitable (fractional) transfer operators similar to \cite{Tud08}. This formula generalizes the results of \cite{Tud08} and \cite{Arr16}. There are two noteworthy properties of  the obtained It\^o-type formula:
	\begin{itemize}
	\itemsep0em
		\item The formula shows that the form \eqref{eq:x_t_intro} of the process $x$ is preserved under compositions with $\mathscr{C}^3$ functions.
		\item There is a term that involves the third derivative.
	\end{itemize}
Both of these properties result from the second-order nature of Rosenblatt processes; that is, from the fact that Rosenblatt processes are defined as second-order Wiener-It\^o integrals.

As suggested in \cite{Arr16}, it seems plausible that the method used to obtain the general It\^o-type formula can be employed to obtain analogous formulas for Hermite processes of any order $k$, see \cite[Definition 3.1]{Tud13}, and it is conjectured to expect the appearance of stochastic integrals with respect to related Hermite processes up to order $k$ (such as $B^{\frac{H}{2}+\frac{1}{2}}$ is related to $R^H$) as well as derivatives up to order $k+1$ in such formulas. Further discussion of this phenomenon can be found on \cite[p. 548]{Arr15} and in \cite[Remark 8]{Tud08}.

The method used to obtain the It\^o-type formula has already been used in the literature, see e.g. \cite{BiaOks08} for the case of fractional Brownian motions and \cite{Tud08, Arr16} for the case of Rosenblatt processes. It can be briefly outlined as follows:
	\begin{enumerate}[label=\emph{Step \arabic*.}, leftmargin=*]
		\item\label{step:intro_1} Initially, two types of integrals with respect to fractional Brownian motions and Rosenblatt processes are defined: a pathwise forward integral defined by regularization of the integrator, see \cite{RusVal93}, and a Skorokhod-type integral defined by means of Malliavin calculus. These definitions are given in \autoref{def:forward_integral}, and in \autoref{def:Skor_int_B} and \autoref{def:Skor_int_R}, respectively. Moreover, the relationship between these two integrals is established. The forward integral with respect to the fractional Brownian motion equals the Skorokhod-type integral plus an additional correction term, see \autoref{prop:relationship_for_Skor_B}. In the case of the Rosenblatt process with Hurst parameter $\sfrac{1}{2}<H<1$, there appear two correction terms and  one of these terms is identified as a Skorokhod-type integral with respect to the fractional Brownian motion of Hurst parameter $\sfrac{H}{2}+\sfrac{1}{2}$, see \autoref{prop:relationship_for_Skor_R}.
			\item The It\^o formula itself is then proved as follows: Starting with a process that has the Skorokhod-type differential \eqref{eq:x_t_intro}, write the integrals as forward integrals by adding the appropriate correction terms using the relationship from \ref{step:intro_1} The important feature of forward integrals is that, unlike the Skorokhod-type integrals, they commute with random variables. Use this property  in the approximation procedure that leads to the It\^o formula and that is based on Taylor's formula. The resulting forward integrals are then rewritten  as Skorokhod-type integrals using results from  \ref{step:intro_1} Finally, convergence of the approximation is verified.
	\end{enumerate}

This paper is organized as follows. In \autoref{sec:prelim}, the general setting is introduced and several notions from Malliavin calculus (Malliavin derivative, Sobolev-Watanabe spaces, and the Skorokhod integral) are reviewed. Furthermore, the definitions of fractional Brownian motions and Rosenblatt processes are also given here.

In \autoref{sec:stoch_int}, some methods that are used throughout the paper but which may also be of independent interest are collected. Initially, the forward integral is recalled which is followed by the definitions of the Skorokhod-type integrals with respect to fractional Brownian motions and Rosenblatt processes. Subsequently, some mapping properties of the Skorokhod-type integrals together with their relationship with the corresponding forward integrals are described. This section is concluded with several technical lemmas useful for computations (product and chain rules for fractional stochastic derivatives, the duality relationship between fractional stochastic derivatives and Skorokhod-type integrals, and a Fubini-type result).

The main results are collected in \autoref{sec:Ito_formulas}. A general It\^o-type formula for functionals of stochastic processes with second-order fractional differential is given in \autoref{prop:Ito_formula}. The assumptions for which the formula is satisfied are kept broad not to limit the possible applications. However, some cases where   the assumptions can be simplified in special cases are described. More precisely, in \autoref{prop:Ito_for_integral_only} and its corollaries, an It\^o formula for functionals of the stochastic integral with respect to a Rosenblatt process which is valid under assumptions formulated in terms of the integrand rather then in terms of the integral is given. This section also contains four technical lemmas which can be useful for computations and three short examples which show how the formulas can be used.

The paper is concluded with \autoref{sec:corollaries} where the second moment of stochastic integral with respect to a Rosenblatt process is computed and where an estimate for its higher absolute moments is given.

\section{Preliminaries}
\label{sec:prelim}

Some notation is briefly described now. The expression $A\lesssim B$ means that that there exists a finite positive constant $c$ such that $A\leq c B$. The constant is independent of the variables which appear in the expressions $A$ and $B$ and it can change from line to line. Similarly, $A\eqsim B$ means that there exists a finite positive constant $c$ such that $A=c B$. These symbols are  used when the exact value of a constant is not important to simplify the exposition. The symbol $\otimes$ denotes the tensor product of two Hilbert spaces or their elements.

\subsection{Some notions from Malliavin calculus}

Let $W=(W_t,t\in\mathbb{R})$ be the two-sided Wiener process defined on a complete probability space $(\Omega,\mathscr{F},\mathbb{P})$ and assume that the $\sigma$-algebra $\mathscr{F}$ is generated by $W$. Several notions of Malliavin calculus that are needed in this paper are now recalled, see e.g. \cite{NouPec12, Nua06} for a complete exposition of the topic.

Let $I(\xi)$ be the first-order Wiener-It\^o integral and let $\mathcal{P}$ be the algebra of polynomial random variables generated by $\{I(\xi), \xi\in L^2(\mathbb{R})\}$, i.e. $\mathcal{P}$ is the set of random variables of the form
		\begin{equation}
		\label{eq:X_polynomial}
		P = p(I(\xi_1), I(\xi_2), \ldots, I(\xi_n))
	\end{equation}
where $p$ is a polynomial, $n\in\mathbb{N}$, and $\xi_i\in L^2(\mathbb{R})$ for every $i=1,2,\ldots, n$. For $k\in\mathbb{N}$ and $P\in\mathcal{P}$ of the form \eqref{eq:X_polynomial}, define the \textit{$k$-th Malliavin derivative of $P$ at point} $x=(x_1,x_2, \ldots,x_k)\in\mathbb{R}^k$ by
	\begin{equation*}
		D^k_xP\overset{\textnormal{Def.}}{=} \sum_{i_1=1}^n\sum_{i_2=1}^n\cdots \sum_{i_k=1}^n\partial_{i_1}\partial_{i_2}\,\cdots\,\partial_{i_k} p(I(\xi_1), I(\xi_2), \ldots ,I(\xi_n))\xi_{i_1}(x_1)\xi_{i_2}(x_2)\cdot\ldots\cdot \xi_{i_k}(x_k).
	\end{equation*}
It follows that for every $k\in\mathbb{N}$ and every $p\geq 1$, the operator $D^k$ is closable from $\mathcal{P}\subset L^p(\Omega)$ to $L^p(\Omega;L^2(\mathbb{R}^k))$, see e.g. \cite[Proposition 2.3.4]{NouPec12}, and its domain is the \textit{Sobolev-Watanabe space} $\mathbb{D}^{k,p}$ defined as the closure of $\mathcal{P}$ with respect to the norm
	\begin{equation*}
		\|P\|_{\mathbb{D}^{k,p}}\overset{\textnormal{Def.}}{=} \left(\mathbb{E}|P|^p + \mathbb{E}\|D P\|_{L^2(\mathbb{R})}^p + \cdots + \mathbb{E}\|D^kP\|_{L^2(\mathbb{R}^k)}^p\right)^\frac{1}{p}.
	\end{equation*}
The operator $D^k$ admits an adjoint denoted by $\delta^k$. Define the set $\mathrm{Dom}(\delta^k)$ as the set of $u\in L^2(\Omega;L^2(\mathbb{R}^k))$ for which there exists a constant $c>0$ such that
	\begin{equation*}
		\left|\left\langle D^k P;u\right\rangle_{L^2(\Omega;L^2(\mathbb{R}^k))}\right|\leq c\|P\|_{L^2(\Omega)}
	\end{equation*}
is satisfied for every $P\in\mathcal{P}$. For $u\in\mathrm{Dom}(\delta^k)$, the symbol $\delta^k(u)$ denotes the unique element of $L^2(\Omega)$ such that
	\begin{equation}
	\label{eq:duality_formula}
		\langle P,\delta^k(u)\rangle_{L^2(\Omega)} = \langle u;D^kP\rangle_{L^2(\Omega;L^2(\mathbb{R}^k))}
	\end{equation}
holds for every $P\in\mathcal{P}$, see e.g. \cite[Definition 2.5.1 and Definition 2.5.2]{NouPec12}. The existence of $\delta^k(u)$ is ensured  by the Riesz representation theorem. The operator $\delta^k:\mathrm{Dom}(\delta^k)\rightarrow L^2(\Omega)$ is called the \textit{(multiple) Skorokhod integral (with respect to $W$)}.

\subsection{Fractional Brownian motions and Rosenblatt processes}

The definitions of a fractional Brownian motion and a Rosenblatt process are now recalled. Denote by $(u)_+ = \max\{u,0\}$ the positive part of $u$.

\begin{definition}
\label{def:fBm}
Let $H\in(\sfrac{1}{2},1)$. The \textit{fractional Brownian motion $B^H=(B^H_t)_{t\in\mathbb{R}}$ of the Hurst parameter $H$} is defined by
	\begin{equation*}
		B_t^H\overset{\textnormal{Def.}}{=} C_H^B\int_{\mathbb{R}}\left(\int_0^t (u-y)_+^{H-\frac{3}{2}}\,\mathrm{d}{u}\right)\,\mathrm{d}{W}_y, \quad t\in\mathbb{R},
	\end{equation*}
where $C_H^B$ is a normalizing constant such that $\mathbb{E}(B_1^H)^2=1$ and the stochastic integral is the Wiener-It\^o integral of order one.
\end{definition}

\begin{definition}
\label{def:Rosenblatt_process}
Let $H\in(\sfrac{1}{2},1)$. The \textit{Rosenblatt process $R^H=(R^H_t)_{t\in\mathbb{R}}$ of the Hurst parameter $H$} is defined by
	\begin{equation*}
		R_t^H\overset{\textnormal{Def.}}{=} C^R_H\int_{\mathbb{R}^2}\left(\int_0^t(u-y_1)_+^{\frac{H}{2}-1}(u-y_2)_+^{\frac{H}{2}-1}\,\mathrm{d}{u}\right)\,\mathrm{d}{W}_{y_1}\,\mathrm{d}{W}_{y_2}, \quad t\in\mathbb{R},
	\end{equation*}
where $C^R_H$ is a normalizing constant such that $\mathbb{E}(R_1^H)^2=1$ and the double stochastic integral is the Wiener-It\^o multiple integral of order two.
\end{definition}

\begin{remark}
\label{rem:constants}
The normalizing constants $C_H^B$ and $C_H^R$ are given by
	\begin{equation*}
			C_H^B = \sqrt{\frac{H(2H-1)}{\mathrm{B}\left(2-2H,H-\frac{1}{2}\right)}},\quad\quad C_H^R = \frac{\sqrt{2H(2H-1)}}{2\mathrm{B}\left(1-H,\frac{H}{2}\right)}
	\end{equation*}
where $\mathrm{B}$ is the Beta function. It will be also convenient to denote
	\begin{equation*}
		c_H^B \overset{\textnormal{Def.}}{=} C_H^B\,\Gamma\left(H-\frac{1}{2}\right), \quad\quad c_H^R \overset{\textnormal{Def.}}{=} C_H^R\,\Gamma\left(\frac{H}{2}\right)^2,
	\end{equation*}
and
	\begin{equation*}
		c_H^{B,R}\overset{\textnormal{Def.}}{=}\frac{c_H^R}{c_{\frac{H}{2}+\frac{1}{2}}^B} = \sqrt{\frac{(2H-1)}{(H+1)}\frac{\Gamma\left(1-\frac{H}{2}\right)\Gamma\left(\frac{H}{2}\right)}{\Gamma(1-H)}}
	\end{equation*}
where $\Gamma$ is the Gamma function.
\end{remark}

\section{Stochastic integration}
\label{sec:stoch_int}
In this section, stochastic integration with respect to fractional Brownian motions and Rosenblatt processes is reviewed and formulas relating the respective forward and Skorokhod integrals are given.

\subsection{Forward-type integrals}
One possible approach to stochastic integration with respect to these processes is via regularization of the integrand in the sense of Russo and Vallois \cite{RusVal93}. This is a natural approach due to the fact that both the processes $B^H$ and $R^H$ have versions with H\"older continuous sample paths of every order smaller than $H$ (recall that $H$ is assumed to be larger than $\sfrac{1}{2}$ in this paper). The definition of the forward integral follows.

\begin{definition}
\label{def:forward_integral}
Let $M\subseteq\mathbb{R}$ be an interval. An integrable stochastic process $g=(g_s)_{s\in M}$ is said to be \textit{forward integrable} on $M$ with respect to a continuous process $h=(h_s)_{s\in\mathbb{R}}$ if the limit in probability
	\begin{equation*}
		\plim_{\varepsilon\downarrow 0}\int_M g_s\frac{h_{s+\varepsilon}-h_s}{\varepsilon}\,\mathrm{d}{s}
	\end{equation*}
exists. The limit is called the \textit{forward integral of $g$ with respect to $h$ on $M$} and it is denoted by $\int_Mg_s\,\mathrm{d}^{-}h_s$.
\end{definition}

\subsection{Skorokhod-type integrals}
On the other hand, stochastic calculus with respect to fractional Brownian motions and Rosenblatt processes can be approached from the perspective of Malliavin calculus. For fractional Brownian motions, such stochastic calculus is now well-developed, e.g. \cite{AlosMazNua01, AlosNua03, DecrUstu99, DunJakDun06}. Here, the definition of the Skorokhod-type integral is recalled and some properties that are useful for the subsequent analysis are reviewed. Note, however, that in the above references, the finite time interval representation of fractional Brownian motions from \cite[Corollary 3.1]{DecrUstu99} is used to define the stochastic integral whereas here, its infinite time interval representation from \autoref{def:fBm} is used as in \cite{Arr15, Arr16, CheNua05,Cou18}.

Consider the first-order fractional integrals defined for $f\in L^p(\mathbb{R})$ by
	\begin{equation*}
		I_+^\alpha (f)(x) = \frac{1}{\Gamma(\alpha)}\int_{-\infty}^xf(u)(x-u)^{\alpha-1}\,\mathrm{d}{u}
	\end{equation*}
and
	\begin{equation*}
		I_{-}^\alpha(f)(x) = \frac{1}{\Gamma(\alpha)}\int_{x}^\infty f(u)(u-x)^{\alpha-1}\,\mathrm{d}{u}
	\end{equation*}
provided that $\alpha\in (0,1)$ and $1\leq p<\sfrac{1}{\alpha}$, cf. \cite[formulas (5.2) and (5.3)]{SamKilMar93}. If the function $f$ takes values in a Banach space, then the integrals above are interpreted as Bochner integrals. The Skorokhod integral with respect to the fractional Brownian motion $B^H$ is defined via the transfer operator $I_{-}^{H-\frac{1}{2}}$ as follows.

\begin{definition}
\label{def:Skor_int_B}
Let $H\in (\sfrac{1}{2},1)$ and let $M\subseteq\mathbb{R}$ be an interval. Define
	\begin{equation*}
		\varLambda_{B^H}(M) \overset{\textnormal{Def.}}{=} \left\{g:\mathbb{R}\rightarrow L^2(\Omega)\mbox{ such that } I_{-}^{H-\frac{1}{2}}(\textbf{1}_Mg)\in \mathrm{Dom}(\delta)\right\}.
	\end{equation*}
A stochastic process $g:\mathbb{R}\rightarrow L^2(\Omega)$ is said to be \textit{Skorokhod integrable with respect to the fractional Brownian motion $B^H$ on $M$} if $g\in\varLambda_{B^H}(M)$. For such integrands, the \textit{Skorokhod integral} is defined by
	\begin{equation*}
		\int_Mg_s\delta B_s^H \overset{\textnormal{Def.}}{=} c_H^B \left(\delta\circ I_-^{H-\frac{1}{2}}\right)\left(\textbf{1}_Mg\right).
	\end{equation*}
\end{definition}

The following proposition provides a mapping property of the Skorokhod integral with respect to the fractional Brownian motion with the Hurst parameter $H\in (\sfrac{1}{2},1)$. In particular, it ensures that stochastic processes from the space $L^\frac{1}{H}(M;\mathbb{D}^{1,2})$ are Skorokhod integrable with respect to the fractional Brownian motion and the stochastic integral is a square-integrable random variable.

\begin{proposition}
\label{prop:boundedness_of_int_B}
Let $H\in(\sfrac{1}{2},1)$ and $M\subseteq\mathbb{R}$ be an interval. The linear operator $\int_M(\cdots)\delta B^H$ is bounded from $L^\frac{1}{H}(M;\mathbb{D}^{k,p})$ to $\mathbb{D}^{k-1,p}$ for every integer $k\geq 1$ and every $p$ such that $1\leq pH<\infty$.
\end{proposition}

\begin{proof}
The proof is similar to the proof of the following  result in \autoref{prop:boundedness_of_int_R}. The main difference between the two is that \cite[Theorem 5.3]{SamKilMar93} is used instead of \autoref{lem:I_tr_bounded} to obtain the estimate
	\begin{equation*}
		\mathbb{E}\left\|I_{-}^{H-\frac{1}{2}}(\bm{1}_MD^lg)\right\|_{L^2(\mathbb{R})\otimes L^2(\mathbb{R}^l)}^p\lesssim \left(\int_M\left(\mathbb{E}\left\|D^lg_s\right\|_{L^2(\mathbb{R}^l)}\right)^\frac{1}{pH}\,\mathrm{d}{s}\right)^{pH}.
	\end{equation*}
\end{proof}

\begin{remark}
For non-fractional Sobolev-Watanabe spaces, \autoref{prop:boundedness_of_int_B} improves \cite[Proposition 16]{Arr16}. There it is shown that the integral $\int_M(\cdots)\delta B^H$ is bounded from $L^2(M;\mathbb{D}^{s,2})$ to $\mathbb{D}^{s-1,2}$ for every real $s\geq 1$ in the white-noise setting, that is, when the probability space is the space of tempered distributions, e.g. \cite{HidaKuoPottStr94}. \autoref{prop:boundedness_of_int_B} is used in the above form for the proofs of the It\^o-type formulas in \autoref{sec:Ito_formulas}.
\end{remark}

The relationship between the forward and the Skorokhod integral with respect to a fractional Brownian motion is given in the next proposition. For the comparison of these integrals, the first-order fractional stochastic derivative $\nabla^{\alpha}$ is defined, following \cite{Arr16}, by
	\begin{equation*}
		\nabla^\alpha\overset{\textnormal{Def.}}{=} I_{+}^\alpha\circ D.
	\end{equation*}
It can be shown that if $\alpha\in (0,\sfrac{1}{2})$, then the operator $\nabla^\alpha$ extends to a bounded linear operator from the Sobolev-Watanabe space $\mathbb{D}^{k,p}$ to the space $L^\frac{2}{1-2\alpha}(\mathbb{R};\mathbb{D}^{k-1,2})$ for $k\geq 1$ and $p\geq 2$ (and consequently to $L^\frac{2}{1-2\alpha}(M;\mathbb{D}^{k-1,2})$ for a bounded interval $M\subset\mathbb{R}$), cf. \cite[Proposition 14]{Arr16}.

\begin{example}
\label{ex:nabla_W}
As an example, an explicit expression for $\nabla^\frac{H}{2}W$ is given where $W$ is the underlying Wiener process and $H\in(\sfrac{1}{2},1)$ is a fixed constant. For every $s\geq 0$, $W_s=I(\bm{1}_{[0,s]})$ and thus
	\begin{align*}
		\nabla^\frac{H}{2}W_s(u) & = \frac{1}{\Gamma\left(\frac{H}{2}\right)}\int_{-\infty}^u (u-x)^{\frac{H}{2}-1}D_x W_s\,\mathrm{d}{x} \\
		& = \frac{1}{\Gamma\left(\frac{H}{2}\right)}\int_{-\infty}^u (u-x)^{\frac{H}{2}-1}\bm{1}_{[0,s]}(x)\,\mathrm{d}{x} \\
		& = \frac{2}{H\Gamma\left(\frac{H}{2}\right)}\left[(u-s)_+^{\frac{H}{2}}-u_+^\frac{H}{2}\right]
	\end{align*}
is satisfied for $s\geq 0$ and $u\in\mathbb{R}$.
\end{example}

The relationship between the forward and the Skorokhod integral with respect to a fractional Brownian motion is given now.

\begin{proposition}
\label{prop:relationship_for_Skor_B}
Let $H\in (\sfrac{1}{2},1)$ and let $g=(g_s)_{s\in [0,T]}$ be a stochastic process such that the following conditions are satisfied:
	\begin{enumerate}
	\itemsep0em
		\item The following integrability is satisfied
			\begin{equation*}
				\int_0^T\|g_s\|_{\mathbb{D}^{1,2}}^\frac{1}{H}\,\mathrm{d}{s} <\infty.
			\end{equation*}
		\item\label{ass:(2)dB} \begin{enumerate}
				 \itemsep0em
						\item For almost every $s\in [0,T]$, the following limit is satisfied
							\begin{equation*}
								\lim_{\varepsilon\downarrow 0}\esssup_{u\in (s,s+\varepsilon)}\|\nabla^{H-\frac{1}{2}}g_s(u)-\nabla^{H-\frac{1}{2}}g_s(s)\|_{L^2(\Omega)}=0.
							\end{equation*}
						\item There exists a non-negative function $p\in L^1(0,T)$ such that, for almost every $s\in [0,T]$, the inequality
							\begin{equation*}
								\|(\nabla^{H-\frac{1}{2}}g_s)(u)\|_{L^2(\Omega)}\leq p(s)
							\end{equation*}
						is satisfied for almost every $u\in [s,T]$.
				 \end{enumerate}
	\end{enumerate}
Then $g$ is forward integrable with respect to $B^H$ on $[0,T]$ and the following equality is satisfied in $L^2(\Omega)$:
	\begin{equation}
	\label{eq:relationship_for_Skor_B}
		\int_0^Tg_s\,\mathrm{d}^{-}B_s^H = \int_0^Tg_s\delta B_s^H + c_H^B\int_0^T(\nabla^{H-\frac{1}{2}}g_s)(s)\,\mathrm{d}{s}.
	\end{equation}
\end{proposition}

\begin{proof}
The proof follows the proof of \cite[Theorem 1]{Arr16}. In particular, instead of the $\varepsilon$-approximation of the forward integral on the left-hand side of \eqref{eq:relationship_for_Skor_B}, its $S$-transform is used which allows to employ the duality \eqref{eq:duality_formula} to determine a new representation of the integral. Then, convergence is verified and the limits are identified as the two terms on the right-hand side of \eqref{eq:relationship_for_Skor_B}. The $S$-transform is given in e.g. \cite[section 2.2]{Ben03a}.
\end{proof}

\begin{remark}
An expression related to formula \eqref{eq:relationship_for_Skor_B} is given in \cite[Theorem 3.7]{BiaOks08} though proved by a different method. The proof here follows the method used in \cite{Arr16} where the equality \eqref{eq:relationship_for_Skor_B} is verified only in the case $g_s=F(B_s^H)$, cf. \cite[Theorem 1]{Arr16}. However, the white-noise setting of \cite{Arr16} is not necessary and it is possible to use a generic probability space as in \cite{Ben03a}.
\end{remark}

The Skorokhod integral with respect to Rosenblatt processes is reviewed now. The definition is similar to \cite[Definition 1]{Tud08} where the finite time interval representation of Rosenblatt processes from \cite[Proposition 1]{Tud08} is used. To simplify computations, the infinite time interval representation given in \autoref{def:Rosenblatt_process} is used here. The Skorokhod integral with respect to a Rosenblatt process is defined via a transfer operator similar to the case of a fractional Brownian motion.

Consider the second-order fractional integral given by
	\begin{equation*}
		(I_{+,+}^{\alpha_1,\alpha_2}f)(x_1,x_2) \overset{\textnormal{Def.}}{=} \frac{1}{\Gamma(\alpha_1)\Gamma(\alpha_2)}\int_{-\infty}^{x_1}\int_{-\infty}^{x_2}f(u,v)(x_1-u)^{\alpha_1-1}(x_2-v)^{\alpha_2-1}\,\mathrm{d}{u}\,\mathrm{d}{v}
	\end{equation*}
for sufficiently nice $f:\mathbb{R}^2\rightarrow\mathbb{R}$ and $\alpha_i\in (0,1)$, $i=1,2$, as in \cite[formula (24.20)]{SamKilMar93}; and define also
\begin{equation*}
	(I_{-,\mathrm{tr}}^{\alpha_1,\alpha_2}f)(x_1,x_2) \overset{\textnormal{Def.}}{=} \frac{1}{\Gamma(\alpha_1)\Gamma(\alpha_2)}\int_{x_1\vee x_2}^\infty f(u)(u-x_1)^{\alpha_1-1}(u-x_2)^{\alpha_2-1}\,\mathrm{d}{u}
\end{equation*}
for sufficiently nice $f:\mathbb{R}\rightarrow \mathbb{R}$. If the function $f$ takes values in a Banach space, then the integrals above are interpreted as Bochner integrals. The operator $I_{-,\mathrm{tr}}^{\frac{H}{2},\frac{H}{2}}$ plays the role of the transfer operator in the following definition of the Skorokhod integral with respect to a Rosenblatt process.

\begin{definition}
\label{def:Skor_int_R}
Let $H\in (\sfrac{1}{2},1)$ and let $M\subseteq\mathbb{R}$ be an interval. Define
	\begin{equation*}
		\varLambda_{R^H}(M) \overset{\textnormal{Def.}}{=} \left\{g:\mathbb{R}\rightarrow L^2(\Omega)\mbox{ such that } I_{-,\mathrm{tr}}^{\frac{H}{2},\frac{H}{2}}(\textbf{1}_Mg)\in \mathrm{Dom}(\delta^2)\right\}.
	\end{equation*}
A stochastic process $g:\mathbb{R}\rightarrow L^2(\Omega)$ is said to be \textit{Skorokhod integrable with respect to the Rosenblatt process $R^H$ on $M$} if $g\in\varLambda_{R^H}(M)$. For such integrands, the \textit{Skorokhod integral} is defined by
	\begin{equation}
	\label{eq:Skor_int_def_Rosenblatt}
		\int_{M}g_s\delta R_s^H \overset{\textnormal{Def.}}{=} c_H^R\,\,\left(\delta^2\circ I_{-,\mathrm{tr}}^{\frac{H}{2},\frac{H}{2}}\right)(\textbf{1}_Mg).
	\end{equation}
\end{definition}

To investigate the mapping properties of the Skorokhod integral with respect to the Rosenblatt process, a technical lemma that is given below is used. Recall the equality
	\begin{equation}
	\label{eq:integral_beta}
		\int_{-\infty}^{u\wedge v} (u-x)^{\alpha -1}(v-x)^{\alpha-1}\,\mathrm{d}{x} = \mathrm{B}(\alpha, 1-2\alpha)|u-v|^{2\alpha-1}, \quad u\neq v,
	\end{equation}
where $\mathrm{B}$ is the Beta function, that is satisfied for $0<\alpha<\sfrac{1}{2}$, see e.g \cite[Lemma 3.1]{Tud13}.

\begin{lemma}
\label{lem:I_tr_bounded}
Let $H\in (\sfrac{1}{2},1)$. The linear operator $I_{-,\mathrm{tr}}^{\frac{H}{2},\frac{H}{2}}$ is bounded from $L^\frac{1}{H}(\mathbb{R})$ to $L^2(\mathbb{R}^2)$.
\end{lemma}

\begin{proof}
Let $g\in L^\frac{1}{H}(\mathbb{R})$. It follows that
	\begin{align*}
		\|I_{-,\mathrm{tr}}^{\frac{H}{2},\frac{H}{2}}g\|_{L^2(\mathbb{R}^2)}^2 & \eqsim \int_{\mathbb{R}^2}\left(\int_{x\vee y}^\infty g(u)(u-x)^{\frac{H}{2}-1}(u-y)^{\frac{H}{2}-1}\,\mathrm{d}{u}\right)^2\,\mathrm{d}{x}\,\mathrm{d}{y}\\
		& \eqsim\int_{\mathbb{R}^2}g(u)g(v)\left(\int_{-\infty}^{u\wedge v}(u-x)^{\frac{H}{2}-1}(v-x)^{\frac{H}{2}-1}\,\mathrm{d}{x}\right)^2\,\mathrm{d}{u}\,\mathrm{d}{v}\\
		& \eqsim\int_{\mathbb{R}^2}g(u)g(v)|u-v|^{2H-2}\,\mathrm{d}{u}\,\mathrm{d}{v}\\
		&\eqsim\int_{\mathbb{R}}g(v)\int_{-\infty}^vg(u)(u-v)^{2H-2}\,\mathrm{d}{u}\,\mathrm{d}{v}\\
		 &\lesssim\left(\int_{\mathbb{R}}|g(u)|^\frac{1}{H}\,\mathrm{d}{u}\right)^{H}\left(\int_{\mathbb{R}}\left|\int_{-\infty}^vg(u)(u-v)^{2H-2}\,\mathrm{d}{u}\right|^\frac{1}{1-H}\,\mathrm{d}{v}\right)^{1-H}\numberthis\label{eq:lem:I_tr_bounded:split_int}
	\end{align*}
where equation \eqref{eq:integral_beta} and H\"older's inequality are used. The rightmost expression in \eqref{eq:lem:I_tr_bounded:split_int} is (modulo a constant) the $L^\frac{1}{1-H}(\mathbb{R})$-norm of $I_+^{2H-1}g$. The use of \cite[Theorem 5.3]{SamKilMar93} gives the estimate
	\begin{equation*}
		\|I_{+}^{2H-1}g\|_{L^\frac{1}{1-H}(\mathbb{R})}\lesssim \|g\|_{L^\frac{1}{H}(\mathbb{R})}.
	\end{equation*}
\end{proof}

The following proposition provides a mapping property of the Skorokhod integral with respect to a Rosenblatt process. It ensures that stochastic processes from the space $L^\frac{1}{H}(M;\mathbb{D}^{2,2})$ are Skorokhod integrable with respect to the Rosenblatt process $R^H$ and the stochastic integral is a square-integrable random variable. Integrability is also treated in \cite[Lemma 1 and Corollary 3]{Tud08}.

\begin{proposition}
\label{prop:boundedness_of_int_R}
Let $H\in(\sfrac{1}{2},1)$ and $M\subseteq\mathbb{R}$ be an interval. The linear operator $\int_M(\cdots)\delta R^H$ is bounded from $L^{\frac{1}{H}}(M;\mathbb{D}^{k,p})$ to $\mathbb{D}^{k-2,p}$ for every integer $k\geq 2$ and every $p$ such that $1\leq pH<\infty$.
\end{proposition}

\begin{proof}
Let $g$ be a $\mathcal{P}$-valued step function. By \cite[Theorem 2.5.5]{NouPec12}, it follows that
	\begin{align*}
		\left\|\int_Mg_s\delta R_s^H\right\|_{\mathbb{D}^{k-2,p}}^p & \eqsim \left\|\delta^2\left(I_{-,\mathrm{tr}}^{\frac{H}{2},\frac{H}{2}}(\bm{1}_{M}g)\right)\right\|_{\mathbb{D}^{k-2,p}}^p \\
		& \lesssim \left\|I_{-,\mathrm{tr}}^{\frac{H}{2},\frac{H}{2}}(\bm{1}_{M}g)\right\|_{\mathbb{D}^{k,p}(L^2(\mathbb{R}^2))}^p \\
		& = \mathbb{E}\left\|I_{-,\mathrm{tr}}^{\frac{H}{2},\frac{H}{2}}(\bm{1}_{M}g)\right\|_{L^2(\mathbb{R}^2)}^p+ \sum_{l=1}^k\mathbb{E}\left\|I_{-,\mathrm{tr}}^{\frac{H}{2},\frac{H}{2}}(\bm{1}_MD^lg)\right\|_{L^2(\mathbb{R}^2)\otimes L^2(\mathbb{R}^l)}^p\numberthis\label{eq:prop:intR:Dnorm_integral}
	\end{align*}
Using \autoref{lem:I_tr_bounded} and Minkowski's inequality \cite[Theorem 202]{HarLittlePol34} twice successively, it follows that
	\begin{align*}
		\mathbb{E}\left\|I_{-,\mathrm{tr}}^{\frac{H}{2},\frac{H}{2}}(\bm{1}_MD^lg)\right\|_{L^2(\mathbb{R}^2)\otimes L^2(\mathbb{R}^l)}^p & = \mathbb{E}\left(\int_{\mathbb{R}^l}\left\|I_{-,\mathrm{tr}}^{\frac{H}{2},\frac{H}{2}}(\bm{1}_{M}D^l_xg)\right\|^2_{L^2(\mathbb{R}^2)}\,\mathrm{d}{x}\right)^\frac{p}{2}\\
		& \lesssim \mathbb{E}\left(\int_{\mathbb{R}^l}\left(\int_M\left|D_x^lg_s\right|^\frac{1}{H}\,\mathrm{d}{s}\right)^{2H}\,\mathrm{d}{x}\right)^\frac{p}{2}\\
		& \lesssim \mathbb{E}\left(\int_M\left(\int_{\mathbb{R}^l}\left|D_x^lg_s\right|^2\,\mathrm{d}{x}\right)^\frac{1}{2H}\,\mathrm{d}{s}\right)^{pH}\\
		& \lesssim \left(\int_M\left(\mathbb{E}\left(\int_{\mathbb{R}^l}\left|D_x^lg_s\right|^2\,\mathrm{d}{x}\right)^\frac{p}{2}\right)^\frac{1}{pH}\,\mathrm{d}{s}\right)^{pH}\\
		& \eqsim \left(\int_M\left(\mathbb{E}\left\|D^lg_s\right\|_{L^2(\mathbb{R}^l)}^p\right)^\frac{1}{pH}\,\mathrm{d}{s}\right)^{pH}.
	\end{align*}
Hence, by using Meyer's inequality \cite[Theorem 1.8]{Wat84} and the embedding $\mathbb{D}^{k,p}\hookrightarrow \mathbb{D}^{l,p}$, it follows that
	\begin{equation}
	\label{eq:prop:intR:estimate_on_Dl}
		\mathbb{E}\left\|I_{-,\mathrm{tr}}^{\frac{H}{2},\frac{H}{2}}(\bm{1}_MD^lg)\right\|_{L^2(\mathbb{R}^2)\otimes L^2(\mathbb{R}^l)}^p \lesssim \left(\int_M \|g_s\|_{\mathbb{D}^{l,p}}^\frac{1}{H}\,\mathrm{d}{s}\right)^{pH}\lesssim \left(\int_M \|g_s\|_{\mathbb{D}^{k,p}}^\frac{1}{H}\,\mathrm{d}{s}\right)^{pH}
	\end{equation}
is satisfied for every integer $l\leq k$. Consequently, it follows that
	\begin{equation*}
		\left\|\int_Mg_s\delta R_s^H\right\|_{\mathbb{D}^{k-2,p}} \lesssim \|g\|_{L^\frac{1}{H}(M;\mathbb{D}^{k,p})}
	\end{equation*}
is satisfied by estimating the terms in \eqref{eq:prop:intR:Dnorm_integral} by \eqref{eq:prop:intR:estimate_on_Dl}. The claim for general $g\in L^\frac{1}{H}(M;\mathbb{D}^{k,p})$ is proved by a standard limit argument.
\end{proof}

The relationship between the forward and the Skorokhod integral with respect to a Rosenblatt process is given in the next proposition. For this relationship, define the second-order fractional stochastic derivative $\nabla^{\alpha,\alpha}$, following \cite{Arr16}, by
	\begin{equation*}
		\nabla^{\alpha,\alpha} \overset{\textnormal{Def.}}{=} I_{+,+}^{\alpha,\alpha}\circ D^2.
	\end{equation*}
It can be shown that if $\alpha\in (0,\sfrac{1}{2})$, the operator $\nabla^{\alpha,\alpha}$ extends to a bounded linear operator from the Sobolev-Watanabe space $\mathbb{D}^{k,p}$ to the space $L^\frac{2}{1-2\alpha}(\mathbb{R}^2;\mathbb{D}^{k-2,2})$ for $k\geq 2$ and $p\geq 2$ (and, consequently, to $L^\frac{2}{1-2\alpha}(M^2;\mathbb{D}^{k-2,2})$ for a bounded interval $M\subset\mathbb{R}$), cf. \cite[Proposition 16]{Arr16}. The relationship between the forward and the Skorokhod integral with respect to a Rosenblatt process follows.

\begin{proposition}
\label{prop:relationship_for_Skor_R}
Let $H\in (\sfrac{1}{2},1)$ and let $g=(g_s)_{s\in [0,T]}$ be a stochastic process such that the following conditions are satisfied:
	\begin{enumerate}
		\item It holds that
			\begin{equation*}
				\int_0^T\|g_s\|_{\mathbb{D}^{2,2}}^\frac{1}{H}\,\mathrm{d}{s} <\infty.
			\end{equation*}
		\item\label{ass:(2)dR} \begin{enumerate}
				  \itemsep0em
				  	\item For almost every $s\in [0,T]$, it follows that
				  		\begin{align*}
				  			& \lim_{\varepsilon\downarrow 0}\esssup_{u\in (s,s+\varepsilon)}\|(\nabla^\frac{H}{2}g_s)(u)-(\nabla^\frac{H}{2}g_s)(s)\|_{\mathbb{D}^{1,2}}=0,\\
				  			& \lim_{\varepsilon\downarrow 0}\esssup_{u\in (s,s+\varepsilon)}\|(\nabla^{\frac{H}{2},\frac{H}{2}}g_s)(u,u)-(\nabla^{\frac{H}{2},\frac{H}{2}}g_s)(s,s)\|_{L^2(\Omega)}=0.
				  		\end{align*}
				  	\item There exist functions $p_1\in L^\frac{2}{1+H}(0,T)$ and $p_2\in L^1(0,T)$ such that, for almost every $s\in [0,T]$, the estimates
				  		\begin{align*}
				  			& \|(\nabla^\frac{H}{2}g_s)(u)\|_{\mathbb{D}^{1,2}}\leq p_1(s),\\
				  			& \|(\nabla^{\frac{H}{2},\frac{H}{2}}g_s)(u,u)\|_{L^2(\Omega)}\leq p_2(s)
				  		\end{align*}
				  	are satisfied for almost every $u\in[s,T]$.
				  \end{enumerate}
	\end{enumerate}
Then $g$ is forward integrable with respect to $R^H$ on $[0,T]$ and the following equality is satisfied in $L^2(\Omega)$:
	\begin{equation}
	\label{eq:relationship_for_Skor_R}
		\int_0^Tg_s\,\mathrm{d}^-R_s^H = \int_0^Tg_s\delta R_s^H + 2c_H^{B,R}\int_0^T (\nabla^\frac{H}{2}g_s)(s)\delta{B}_s^{\frac{H}{2}+\frac{1}{2}} + c_H^R\int_0^T(\nabla^{\frac{H}{2},\frac{H}{2}}g_s)(s,s)\,\mathrm{d}{s}.
	\end{equation}
\end{proposition}

\begin{proof}
The proof follows similarly to the proof of \cite[Theorem 3]{Arr16} where equality \eqref{eq:relationship_for_Skor_R} is only proved in the case $g_s=F(R_s^H)$.
\end{proof}

\begin{remark}
Equality \eqref{eq:relationship_for_Skor_R} should be compared to the equality given in \cite[Theorem 2]{Tud08} where two correction terms arise as a result of the integration by parts formula for the double Skorokhod integral $\delta^2$, see \cite[formula (34)]{Tud08} or \cite{NuaZak87}. These two terms have a similar structure as the ones in expression \eqref{eq:relationship_for_Skor_R} and they are given in terms of two limits, called the \textit{trace of order 1} and \textit{trace of order 2}. The result in \cite[Theorem 2]{Tud08} is then proved under the assumption that these limits exist. In \autoref{prop:relationship_for_Skor_R}, sufficient conditions for this convergence are given, and, moreover, the term that corresponds to the trace of order 1 is identified as the Skorokhod integral with respect to the fractional Brownian motion of the Hurst parameter $\sfrac{H}{2} + \sfrac{1}{2}$.
\end{remark}

\subsection{Additional lemmas}

This section contains several results that are useful in the rest of the paper and that can also be useful for applications. Initially, a product rule for the fractional stochastic derivatives $\nabla^\alpha$ and $\nabla^{\alpha,\alpha}$ is given.

\begin{lemma}
\label{lem:product_rule}
Let $0<\alpha<\sfrac{1}{2}$.
\begin{enumerate}[label=\arabic*)]
	\item\label{lem:prod_H_1} If $G_1,G_2\in\mathbb{D}^{1,4}$, then $G_1G_2\in\mathbb{D}^{1,2}$ and
	\begin{equation*}
		\nabla^{\alpha}(G_1G_2) = (\nabla^\alpha G_1)G_2 + G_1(\nabla^\alpha G_2).
	\end{equation*}
	\item\label{lem:prod_H_2} If $G_1,G_2\in\mathbb{D}^{2,4}$, then $G_1G_2\in\mathbb{D}^{2,2}$ and
	\begin{equation*}
		\nabla^{\alpha,\alpha}(G_1G_2) = (\nabla^{\alpha,\alpha}G_1)G_2 + 2(\nabla^\alpha G_1)\tilde{\otimes}(\nabla^\alpha G_2) + G_1(\nabla^{\alpha,\alpha}G_2)
	\end{equation*}
	where $\tilde{\otimes}$ denotes the symmetrization of the tensor product $\otimes$.
\end{enumerate}
\end{lemma}

\begin{proof}
The proof follows directly from \cite[Exercise 2.3.10]{NouPec12}.
\end{proof}

The following chain rule for the Malliavin derivative implies a chain rule for the fractional stochastic derivatives. Furthermore, its proof gives estimates that are used in the proof of \autoref{prop:Ito_for_integral_only}.

\begin{lemma}
\label{lem:chain_rule}
Let $p>1$.
\begin{enumerate}[label=\arabic*)]
	\item\label{lem:chain_1} Let $f\in\mathscr{C}^1(\mathbb{R})$ be such that
\begin{equation*}
	\label{eq:f'_poly_growth}
		|f'(x)|\leq c(1+|x|^\beta), \quad x\in\mathbb{R},
	\end{equation*}
is satisfied for some $c\geq 0$ and $\beta\geq 0$. If $G\in\mathbb{D}^{1,p(\beta+1)}$, then $f(G)\in\mathbb{D}^{1,p}$ and
		\begin{equation*}
			D f(G) = f'(G)DG.
		\end{equation*}
	\item\label{lem:chain_2} Let $f\in\mathscr{C}^2(\mathbb{R})$ be such that
\begin{equation}
	\label{eq:f''_poly_growth}
		|f''(x)|\leq c(1+|x|^\beta), \quad x\in\mathbb{R},
	\end{equation}
is satisfied for some $c\geq 0$ and $\beta\geq 0$. If $G\in\mathbb{D}^{2,p(\beta+2)}$, then $f(G)\in\mathbb{D}^{2,p}$ and
		\begin{equation}
		\label{eq:chain_2}
			D^2f(G) = f''(G)D G\otimes D G + f'(G)D^2 G.
		\end{equation}
\end{enumerate}
\end{lemma}

\begin{proof}
The proof of \ref{lem:chain_1} is given in \cite{NuaNua18} and \ref{lem:chain_2} is proved similarly. Initially, equality \eqref{eq:chain_2} is verified for $f$ smooth and $G\in\mathcal{P}$. The result is then extended for general $f$ and $G$ by approximation using the estimates
\begin{equation}
\label{eq:estimates_on_norms_of_f}
		\|f(G)\|_{\mathbb{D}^{1,p}} \lesssim(1+\|G\|_{\mathbb{D}^{1,p(\beta+2)}}^{\beta+2})\quad\mbox{and}\quad \|f(G)\|_{\mathbb{D}^{2,p}} \lesssim (1+\|G\|_{\mathbb{D}^{2,p(\beta+2)}}^{\beta+2}),
	\end{equation}
which follow by H\"older's inequality because there exist constants $c', c''\geq 0$ such that
\begin{equation*}
		|f'(x)|\leq c'(1+|x|^{\beta+1})\quad\mbox{and}\quad |f(x)|\leq c''(1+|x|^{\beta+2}).
	\end{equation*}
\end{proof}

\begin{corollary}
\label{cor:chain_rule_H}
Let $0<\alpha<\sfrac{1}{2}$ and $p>1$.
	\begin{enumerate}[label=\arabic*)]
		\item\label{cor:chain_1_H} Let $f$ be as in \ref{lem:chain_1} of \autoref{lem:chain_rule}. If $G\in\mathbb{D}^{1,p(\beta+2)}$, then
			\begin{equation*}
				\nabla^{\alpha}f(G) = f'(G)(\nabla^\alpha G).
			\end{equation*}
		\item\label{cor:chain_2_H} Let $f$ be as in \ref{lem:chain_2} of \autoref{lem:chain_rule}. If $G\in\mathbb{D}^{2,p(\beta+2)}$, then 	
			\begin{equation*}
				\nabla^{\alpha, \alpha}f(G) = f''(G)(\nabla^\alpha G)\otimes(\nabla^\alpha G) + f'(G)(\nabla^{\alpha,\alpha}G).
			\end{equation*}
	\end{enumerate}
\end{corollary}

\begin{proof}
The proof of \ref{cor:chain_1_H} follows directly from \ref{lem:chain_1} of \autoref{lem:chain_rule}. The proof of \ref{cor:chain_2_H} follows from \ref{lem:chain_2} of \autoref{lem:chain_rule} using the fact that for $f(u,v)=f_1(u) f_2(v)$, the double integral $I_{+,+}^{\alpha,\alpha}(f)(u,v)$ equals $ I_+^\alpha(f_1)(u) I_+^\alpha (f_2)(v)$.
\end{proof}

The following lemma relates the Skorokhod integrals with respect to the fractional Brownian motion $B^H$ and the Rosenblatt process $R^H$ to the fractional stochastic derivatives $\nabla^{H-\frac{1}{2}}$ and $\nabla^{\frac{H}{2},\frac{H}{2}}$, respectively.

\begin{lemma} Let $H\in (\sfrac{1}{2},1)$ and $M\subset\mathbb{R}$ be an interval.
\label{lem:adjoint_property}
	\begin{enumerate}[label=\arabic*)]
		\item\label{lem:adjoint_property_B} If $g\in L^\frac{1}{H}(M;\mathbb{D}^{1,2})$ and $G\in\mathbb{D}^{1,2}$, then the following equality is satisfied:
			\begin{equation*}
				\mathbb{E}\left[ G\int_Mg_u\delta B_u^H\right]= c_H^{B}\int_M\mathbb{E}\left[(\nabla^{H-\frac{1}{2}}G)(u)g_u\right]\,\mathrm{d}{u}.
			\end{equation*}
		\item\label{lem:adjoint_property_R} If $g\in L^\frac{1}{H}(M;\mathbb{D}^{2,2})$ and $G\in\mathbb{D}^{2,2}$, then the following equality is satisfied:
			\begin{equation*}
				\mathbb{E}\left[ G\int_Mg_u\delta R_u^H\right]= c_H^R\int_M\mathbb{E}\left[ (\nabla^{\frac{H}{2},\frac{H}{2}}G)(u,u)g_u\right]\,\mathrm{d}{u}.
			\end{equation*}
	\end{enumerate}
\end{lemma}

\begin{proof}
The adjoint property in \ref{lem:adjoint_property_B} follows directly from the duality formula \eqref{eq:duality_formula} and the integration by parts formula for fractional integrals $I_{-}^\alpha$ and $I_{+}^\alpha$, see \cite[formula (5.16) on p. 96]{SamKilMar93}. \cite[Proposition 18]{Arr16} provides the proof. To prove \ref{lem:adjoint_property_R}, use the duality formula \eqref{eq:duality_formula} and interchange the order of the integrals. Thus
	\begin{align*}
		\mathbb{E}\left[ G\int_Mg_s\delta R_s^H\right] & = c_H^R\left\langle G;\delta ^2\left(I_{-,\mathrm{tr}}^{\frac{H}{2},\frac{H}{2}}(\bm{1}_Mg)\right)\right\rangle_{L^2(\Omega)}\\
		& = c_H^R\left\langle D^2G;I_{-,\mathrm{tr}}^{\frac{H}{2},\frac{H}{2}}(\bm{1}_Mg)\right\rangle_{L^2(\mathbb{R}^2;L^2(\Omega))}\\
		& = \frac{c_H^R}{\Gamma\left(\frac{H}{2}\right)^2}\int_{\mathbb{R}^2}\mathbb{E}\left[D^2_{x,y}G\int_{x\vee y}^\infty\bm{1}_M(u)g_u(u-x)^{\frac{H}{2}-1}(u-y)^{\frac{H}{2}-1}\,\mathrm{d}{u}\right]\,\mathrm{d}{x}\,\mathrm{d}{y}\\
		& = \frac{c_H^R}{\Gamma\left(\frac{H}{2}\right)^2}\int_M \mathbb{E} \left[g_u\int_{-\infty}^u\int_{-\infty}^uD_{x,y}^2G(u-x)^{\frac{H}{2}-1}(u-y)^{\frac{H}{2}-1}\,\mathrm{d}{x}\,\mathrm{d}{y}\right]\,\mathrm{d}{u}\\
		& = c_H^R\int_M\mathbb{E}\left[ (\nabla^{\frac{H}{2},\frac{H}{2}}G)(u,u)g_u\right]\,\mathrm{d}{u}.
	\end{align*}
The use of Fubini's theorem follows by the estimate
	\begin{align*}
	\frac{c_H^R}{\Gamma\left(\frac{H}{2}\right)^2}\int_{\mathbb{R}^2}\mathbb{E}\left[|D^2_{x,y}G|\int_{x\vee y}^\infty\bm{1}_M(u)|g_u|(u-x)^{\frac{H}{2}-1}(u-y)^{\frac{H}{2}-1}\,\mathrm{d}{u}\right]\,\mathrm{d}{x}\,\mathrm{d}{y} & \lesssim \\
	& \hspace{-4cm} \lesssim \|D^2G\|_{L^2(\mathbb{R}^2;L^2(\Omega))}\|I_{-,\mathrm{tr}}^{\frac{H}{2},\frac{H}{2}}(\bm{1}_Mg)\|_{L^2(\mathbb{R}^2;L^2(\Omega))}\\
	& \hspace{-4cm} \lesssim\|G\|_{\mathbb{D}^{2,2}}\|g\|_{L^\frac{1}{H}(M;L^2(\Omega))}
	\end{align*}
which follows by H\"older's inequality and \cite[Proposition 2.5.5]{NouPec12} with \autoref{lem:I_tr_bounded}.
\end{proof}

This section ends with a Fubini theorem for the Skorokhod integral with respect to Rosenblatt processes.

\begin{lemma}
\label{lem:fubini}
Let $H\in (\sfrac{1}{2},1)$ and let $(E,\mu)$ be a measurable space equipped with a finite positive measure $\mu$. Let further $g:[0,T]\times E\rightarrow \mathbb{D}^{2,2}$ be a random field such that
	\begin{equation*}
		\int_E\left(\int_0^T\|g(s,x)\|_{\mathbb{D}^{2,2}}^\frac{1}{H}\,\mathrm{d}{s}\right)^{2H}\mu(\mathrm{d}{x}) <\infty.
	\end{equation*}
Then the equality
	\begin{equation*}
		\int_E\left(\int_0^T g(s,x)\delta R_s^H\right)\mu(\mathrm{d}{x}) = \int_0^T\left(\int_Eg(s,x)\mu(\mathrm{d}{x})\right)\delta R_s^H
	\end{equation*}
is satisfied almost surely.
\end{lemma}

\begin{proof}
By \cite[Proposition 2.6]{NuaZak87}, the Fubini-type theorem for the Skorokhod integral $\delta$ \cite[Lemma 2.10]{NuaLeo98} can be used iteratively.
\end{proof}

\section{A general It\^o formula}
\label{sec:Ito_formulas}

In this section, a generalized It\^o-type formula for functionals of processes with second-order fractional differentials is given. Moreover, several results useful for its applications are given here as well. Let $H\in (\sfrac{1}{2},1)$ and $T>0$ be fixed for the remainder of the paper.

\begin{proposition}
\label{prop:Ito_formula}
Let $f\in\mathscr{C}^2([0,T]\times\mathbb{R})$ be such that for every $t\in [0,T]$, the function $f(t,\cdot)$ belongs to $\mathscr{C}^3(\mathbb{R})$ and its third derivative has at most polynomial growth, i.e. there exist constants $C_t\geq 0$ and $\alpha\geq 0$ such that
	\begin{equation*}
		\left|\frac{\partial^3 f}{\partial x^3}(t,x)\right| \leq C_t(1+|x|^\alpha), \quad x\in\mathbb{R}.
	\end{equation*}
Let $x_0\in\mathbb{R}$ and $(\vartheta_t)_{t\in[0,T]}$, $(\varphi_t)_{t\in [0,T]}$, and $(\psi_t)_{t\in [0,T]}$ be stochastic processes satisfying the following:
	\begin{enumerate}[label=\upshape{(X\arabic*)}]
	\itemsep0em
		\item\label{ass:gen_1} $(\vartheta_t)$ belongs to the space $L^1(0,T;\mathbb{D}^{2,2(\alpha+2)})$.
		\item $(\varphi_t)$ belongs to the space $L^\frac{2}{1+H}(0,T;\mathbb{D}^{3,2(\alpha+2)})$ and satisfies the following conditions:
			\begin{enumerate}
				\itemsep0em
				\item For almost every $u\in[0,T]$, the following equality is satisfied
					\begin{equation*}
						\lim_{\varepsilon\downarrow0}\esssup_{v\in (u,u+\varepsilon)}\|(\nabla^\frac{H}{2}\varphi_u)(v) - (\nabla^\frac{H}{2}\varphi_u)(u)\|_{L^4(\Omega)} = 0.
					\end{equation*}
 				\item There exists a function $p_1\in L^1(0,T)$ such that, for almost every $u\in [0,T]$, the estimate
					\begin{equation*}
						\|\nabla^\frac{H}{2}\varphi_u(v)\|_{L^4(\Omega)}\leq p_1(u)
					\end{equation*}
				is satisfied for almost every $v\in [u,T]$.
			\end{enumerate}
		\item\label{ass:gen_3} $(\psi_t)$ belongs to the space $L^\frac{1}{H}(0,T;\mathbb{D}^{4,4(\alpha+2)})$ and satisfies the following conditions:
			\begin{enumerate}
			\itemsep0em
				\item For almost every $u\in [0,T]$, it holds that
					\begin{align*}
						& \lim_{\varepsilon\downarrow 0}\esssup_{v\in (u,u+\varepsilon)}\|(\nabla^\frac{H}{2}\psi_u)(v)-(\nabla^\frac{H}{2}\psi_u)(u)\|_{\mathbb{D}^{1,8}}=0,\\
						& \lim_{\varepsilon\downarrow 0}\esssup_{v\in (u,u+\varepsilon)}\|(\nabla^{\frac{H}{2},\frac{H}{2}}\psi_u)(v,u) -(\nabla^{\frac{H}{2},\frac{H}{2}}\psi_u)(u,u)\|_{L^4(\Omega)} = 0,\\
						& \lim_{\varepsilon\downarrow 0}\esssup_{v\in (u,u+\varepsilon)}\|(\nabla^{\frac{H}{2},\frac{H}{2}}\psi_u)(v,v) -(\nabla^{\frac{H}{2},\frac{H}{2}}\psi_u)(u,u)\|_{L^4(\Omega)} = 0.
					\end{align*}
				\item There exist functions $p_2\in L^{\frac{2}{1+2H}}(0,T)$, $p_3\in L^1(0,T)$, and $p_4\in L^1(0,T)$ such that, for almost every $u\in [0,T]$, the inequalities
				\begin{align*}
					& \|(\nabla^\frac{H}{2}\psi_u)(v)\|_{\mathbb{D}^{1,8}}\leq p_2(u),\\
					& \|(\nabla^{\frac{H}{2},\frac{H}{2}}\psi_u)(u,v)\|_{L^4(\Omega)} \leq p_3(u),\\
					& \|(\nabla^{\frac{H}{2},\frac{H}{2}}\psi_u)(v,v)\|_{L^4(\Omega)} \leq p_4(u),
				\end{align*}
			are satisfied for almost every $v\in [u,T]$.
			\end{enumerate}		
	\end{enumerate}
Define the stochastic process $(x_t)_{t\in [0,T]}$ by
	\begin{equation}
	\label{eq:fractional_differential}
		x_t \overset{\mathrm{Def.}}{=} x_0+\int_0^t\vartheta_s\,\mathrm{d}{s} + 2c_H^{B,R}\int_0^t\varphi_s\delta B_s^{\frac{H}{2}+\frac{1}{2}}+\int_0^t\psi_s\delta{R}^H_s
	\end{equation}
and assume that it is H\"older continuous of an order greater than $\sfrac{1}{2}$. Furthermore assume  the following:
	\begin{enumerate}[label=\upshape{(X\arabic*)}]
	\setcounter{enumi}{3}
		\item\label{ass:gen_4} The following finiteness conditions are satisfied
			\begin{equation*}
				 \esssup_{s\in [0,T]} \left\|\frac{\partial f}{\partial x}(s,x_s)\right\|_{\mathbb{D}^{2,4}} <\infty, \quad \esssup_{s\in [0,T]} \left\|\frac{\partial^2 f}{\partial x^2}(s,x_s)\right\|_{\mathbb{D}^{1,8}} <\infty, \quad \esssup_{s\in [0,T]} \left\|\frac{\partial^3 f}{\partial x^3}(s,x_s)\right\|_{L^8(\Omega)} <\infty.
			\end{equation*}
		\item\label{ass:gen_5} For almost every $v\in [0,T]$, the following equalities are satisfied
			\begin{align*}
				& \lim_{\varepsilon\downarrow 0}\esssup_{u\in (v-\varepsilon,v)} \|(\nabla^{\frac{H}{2}}X_u)(v) - (\nabla^\frac{H}{2}X_v)(v)\|_{\mathbb{D}^{1,8}}=0,\\
				&\lim_{\varepsilon\downarrow 0}\esssup_{u\in (v-\varepsilon,v)} \|(\nabla^{\frac{H}{2}, \frac{H}{2}}X_u)(v,v) - (\nabla^\frac{H}{2}X_v)(v,v)\|_{L^8(\Omega)}=0.\\
			\end{align*}
		\item\label{ass:gen_6} For almost every $s\in [0,T]$ and almost every $v\in [0,T]$, the following equalities are satisfied
			\begin{align*}
				& \lim_{\varepsilon\downarrow 0} \esssup_{u\in (v,v+\varepsilon)}\|(\nabla^\frac{H}{2}X_s)(u) - (\nabla^\frac{H}{2}X_s)(v)\|_{\mathbb{D}^{1,8}} = 0,\\
				& \lim_{\varepsilon\downarrow 0} \esssup_{u\in (v,v+\varepsilon)}\|(\nabla^{\frac{H}{2},\frac{H}{2}}X_s)(u,u) - (\nabla^{\frac{H}{2},\frac{H}{2}}X_s)(v,v)\|_{L^{8}(\Omega)} = 0.
			\end{align*}
		\item\label{ass:gen_7} The following finiteness conditions are satisfied
			\begin{equation*}
				\esssup_{s,u\in [0,T]} \|(\nabla^\frac{H}{2}X_s)(u)\|_{\mathbb{D}^{1,8}}<\infty \quad \mbox{and} \quad \esssup_{s,u\in [0,T]}\|(\nabla^{\frac{H}{2},\frac{H}{2}}X_s(u,u)\|_{L^8(\Omega)}<\infty.
			\end{equation*}
	\end{enumerate}
Then for the process $(y_t)_{t\in [0,T]}$ defined by $y_t=f(t,x_t)$ the equality 
	\begin{equation}
	\label{eq:Ito_formula_full}
		y_t = y_0 + \int_0^t\tilde{\vartheta}_s\,\mathrm{d}{s} + 2c_H^{B,R}\int_0^t\tilde{\varphi}_s\delta B_s^{\frac{H}{2}+\frac{1}{2}} + \int_0^t\tilde{\psi}_s\delta R^H_s
	\end{equation}
is satisfied with
	\begin{align*}
		\tilde{\vartheta}_s & = \frac{\partial f}{\partial s}(s,x_s) + \frac{\partial f}{\partial x}(s,x_s)\vartheta_s  \\
		& \hspace{2cm} +\,\, 2c_H^R\frac{\partial^2 f}{\partial x^2}(s,x_s)(\nabla^\frac{H}{2}x_s)(s)\varphi_s\\
			& \hspace{4cm} + \,\, c_H^R \frac{\partial^2f}{\partial x^2}(s,x_s) (\nabla^{\frac{H}{2},\frac{H}{2}}x_s)(s,s)\psi_s \\
			& \hspace{6cm} + \,\, c_H^R\frac{\partial^3f}{\partial x^3}(s,x_s) [(\nabla^\frac{H}{2}x_s)(s)]^2\psi_s,\\
			\tilde{\varphi}_s & = \frac{\partial f}{\partial x}(s,x_s)\varphi_s + \frac{\partial^2 f}{\partial x^2}(s,x_s)(\nabla^\frac{H}{2}x_s)(s)\psi_s,\\
			\tilde{\psi}_s & = \frac{\partial f}{\partial x}(s,x_s)\psi_s.
	\end{align*}
for every $t\in [0,T]$ almost surely.
\end{proposition}

\begin{proof}
Let $t\in[0,T]$. Initially, the proof of the It\^o formula for the forward integral in \cite[Theorem 1.2]{RusVal95} is used. Since the process $(x_t)$ has continuous sample paths, the equality
	\begin{equation*}
		f(t,x_t)-f(0,x_0) = \lim_{\varepsilon\downarrow 0} \frac{1}{\varepsilon}\int_0^t\left[f(s+\varepsilon,x_{s+\varepsilon}) - f(s,x_s)\right]\,\mathrm{d}{s} \overset{\mathrm{Def.}}{=} \lim_{\varepsilon\downarrow 0} A_{t,\varepsilon}
	\end{equation*}
is satisfied almost surely. Using Taylor's formula to expand the integrand in $A_{t,\varepsilon}$ yields, for $\varepsilon>0$, that
	\begin{align*}
		A_{t,\varepsilon} & = \int_0^t\frac{\partial f}{\partial x}(s,x_s)\,\mathrm{d}{s}+ \frac{1}{\varepsilon}\int_0^t\frac{\partial f}{\partial x}(s,x_s)(x_{s+\varepsilon}-x_s)\,\mathrm{d}{s}\\
		& \hspace{1cm} + \varepsilon \int_0^tR_1(s,\varepsilon)\,\mathrm{d}{s} + \int_0^tR_2(s,\varepsilon)(x_{s+\varepsilon}-x_s)\,\mathrm{d}{s} + \frac{1}{\varepsilon}\int_0^tR_3(s,\varepsilon)(x_{s+\varepsilon}-x_s)^2\,\mathrm{d}{s}\numberthis\label{eq:Ito_4}
	\end{align*}
is satisfied almost surely. Here, the processes $R_1,R_2,R_3$ are remainders in the integral form. All the three terms containing these remainders tend to zero as $\varepsilon \downarrow 0$; the second by continuity of $x$, the third by the fact that sample paths of $x$ are H\"older continuous of an order greater than $\sfrac{1}{2}$.

Therefore focus on the second term of $A_{t,\varepsilon}$ that tends to the forward integral. An idea from \cite{BiaOks08} is used. By \autoref{prop:relationship_for_Skor_B} and \autoref{prop:relationship_for_Skor_R}, the Skorokhod-type integrals in $x$ can be written as forward integrals by adding the appropriate correction terms:
	\begin{align*}
		x_{s+\varepsilon} -x_s & =\int_{s}^{s+\varepsilon}\psi_u\,\mathrm{d}^{-}R_u^H + 2c_{H}^{B,R}\int_s^{s+\varepsilon} \left[\varphi_u-(\nabla^\frac{H}{2}\psi_u)(u)\right]\,\mathrm{d}^{-}{B}_u^{\frac{H}{2}+\frac{1}{2}}\\
		& \hspace{1.5cm} + \int_s^{s+\varepsilon}\left[\vartheta_u - 2c_H^R(\nabla^\frac{H}{2}\varphi_u)(u)+c_H^R(\nabla^{\frac{H}{2},\frac{H}{2}}\psi_u)(u,u)\right]\,\mathrm{d}{u}.
	\end{align*}
Forward integrals commute with random variables so that the equality
	\begin{align*}
		\frac{\partial f}{\partial x}(s,x_s)(x_{s+\varepsilon}-x_s) & = \int_s^{s+\varepsilon} \frac{\partial f}{\partial x}(s,x_s)\psi_u\,\mathrm{d}^{-}R_u^H \\
		& \hspace{1cm} + 2c_H^{B,R}\int_s^{s+\varepsilon}\frac{\partial f}{\partial x}(s,x_s)\left[\varphi_u-(\nabla^\frac{H}{2}\psi_u)(u)\right]\,\mathrm{d}^{-}{B}_u^{\frac{H}{2}+\frac{1}{2}}\\
		& \hspace{1cm} + \int_s^{s+\varepsilon} \frac{\partial f}{\partial x}(s,x_s)\left[\vartheta_u - 2c_H^R(\nabla^\frac{H}{2}\varphi_u)(u)+c_H^R(\nabla^{\frac{H}{2},\frac{H}{2}}\psi_u)(u,u)\right]\,\mathrm{d}{u}
	\end{align*}
is satisfied almost surely. Denote
	\begin{equation*}
		B_{s,\varepsilon} \overset{\mathrm{Def.}}{=} \frac{\partial f}{\partial x}(s,x_s)(x_{s+\varepsilon}-x_s).
	\end{equation*}
By \autoref{prop:relationship_for_Skor_B} and \autoref{prop:relationship_for_Skor_R}, the forward integrals above can be written as Skorokhod-type integrals and thus
		\begin{align*}
		B_{s,\varepsilon} & = \int_s^{s+\varepsilon} \frac{\partial f}{\partial x}(s,x_s)\psi_u\delta R_u^H \\
		& \hspace{1cm} + 2c_H^{B,R}\int_s^{s+\varepsilon}\left[\frac{\partial f}{\partial x}(s,x_s)\left(\varphi_u-(\nabla^\frac{H}{2}\psi_u)(u)\right)+\nabla^\frac{H}{2}\left(\frac{\partial f}{\partial x}(s,x_s)\psi_u\right)(u)\right]\delta{B}_u^{\frac{H}{2}+\frac{1}{2}}\\
		& \hspace{1cm} + \int_s^{s+\varepsilon}\bigg[\frac{\partial f}{\partial x}(s,x_s)\vartheta_u+ 2c_H^R\nabla^\frac{H}{2}\left(\frac{\partial f}{\partial x}(s,x_s)\varphi_u\right)(u)-2c_H^R\frac{\partial f}{\partial x}(s,x_s)(\nabla^\frac{H}{2}\varphi_u)(u) \\
		& \hspace{3cm} + c_H^R\frac{\partial f}{\partial x}(s,x_s)\nabla^{\frac{H}{2},\frac{H}{2}}\psi_u(u,u) + c_H^R\nabla^{\frac{H}{2},\frac{H}{2}}\left(\frac{\partial f}{\partial x}(s,x_s)\psi_u\right)(u,u)\\
		&\hspace{3cm} -2c_H^R\nabla^\frac{H}{2}\left(\frac{\partial f}{\partial x}(s,x_s)\nabla^\frac{H}{2}\psi_u(u)\right)(u)\bigg]\,\mathrm{d}{u}.\numberthis\label{eq:Ito_5}
	\end{align*}
The terms in the above expression are now simplified. Using the product and chain rules from \autoref{lem:product_rule} and \autoref{cor:chain_rule_H}, it follows that
	\begin{equation*}
		\nabla^\frac{H}{2}\left(\frac{\partial f}{\partial x}(s,x_s)\psi_u\right)(u) = \frac{\partial^2f}{\partial x^2}(s,x_s)(\nabla^\frac{H}{2} x_s)(u)\psi_u + \frac{\partial f}{\partial x}(s,x_s)(\nabla^\frac{H}{2} \psi_u)(u).
	\end{equation*}
Similarly, it follows that
	\begin{align*}
		\nabla^\frac{H}{2}\left(\frac{\partial f}{\partial x}(s,x_s)\nabla^\frac{H}{2}\psi_u(u)\right)(u) = \frac{\partial^2f}{\partial x}(s,x_s)\nabla^\frac{H}{2}x_s(u)\nabla^\frac{H}{2}\psi_u(u) + \frac{\partial f}{\partial x}(s,x_s)\nabla^{\frac{H}{2},\frac{H}{2}}\psi_u(u,u)
	\end{align*}
and, for the second-order term, that
		\begin{align*}
			\nabla^{\frac{H}{2},\frac{H}{2}}\left(\frac{\partial f}{\partial x}(s,x_s)\psi_u\right)(u,u) & = \frac{\partial^3f}{\partial x^3}(s,x_s)\left[(\nabla^\frac{H}{2}x_s)(u)\right]^2\psi_u + \\
			& \hspace{1cm} + \frac{\partial^2f}{\partial x^2}(s,x_s) \left[2(\nabla^\frac{H}{2}\psi_u)(u)(\nabla^\frac{H}{2}x_s)(u) + \psi_u(\nabla^{\frac{H}{2},\frac{H}{2}}x_s)(u,u) \right]\\
			& \hspace{1cm} + \frac{\partial f}{\partial x}(s,x_s)(\nabla^{\frac{H}{2},\frac{H}{2}}\psi_u)(u,u).
		\end{align*}
Substituting these formulas in \eqref{eq:Ito_5} yields that
	\begin{align*}
		B_{s,\varepsilon} & = \int_s^{s+\varepsilon}\frac{\partial f}{\partial x}(s,x_s)\psi_u\delta R_u^H \\
		& \hspace{1cm} + 2c_H^{B,R}\int_s^{s+\varepsilon}\left[\frac{\partial f}{\partial x}(s,x_s)\varphi_u+ \frac{\partial^2f}{\partial x^2}(s,x_s)(\nabla^\frac{H}{2} x_s)(u)\psi_u\right]\delta{B}_u^{\frac{H}{2}+\frac{1}{2}}\\
		& \hspace{1cm} + \int_s^{s+\varepsilon}\bigg[\frac{\partial f}{\partial x}(s,x_s)\vartheta_u + 2c_H^R\frac{\partial^2f}{\partial x^2}(s,x_s)(\nabla^\frac{H}{2}x_s)(u)\varphi_u \\
		& \hspace{3cm} + c_H^R\frac{\partial^2 f}{\partial x^2}(s,x_s)(\nabla^{\frac{H}{2},\frac{H}{2}}x_s)(u,u)\psi_u +c_H^R\frac{\partial^3f}{\partial x^3}(s,x_s)[(\nabla^\frac{H}{2}x_s)(u)]^2\psi_u\bigg]\,\mathrm{d}{u}
	\end{align*}
is satisfied almost surely. In order to finish the proof, it suffices to verify the convergence
	\begin{equation}
	\label{eq:Ito_3}
		\frac{1}{\varepsilon}\int_0^tB_{s,\varepsilon}\,\mathrm{d}{s} \quad \overset{L^2(\Omega)}{\underset{\varepsilon\downarrow 0}{\longrightarrow }}\quad \int_0^t \left[\tilde{\vartheta}_s-\frac{\partial f}{\partial s}(s,x_s)\right]\,\mathrm{d}{s} + 2c_H^{B,R}\int_0^t\tilde{\varphi}_s\delta B_s^{\frac{H}{2}+\frac{1}{2}} + \int_0^t \tilde{\psi}_s\delta R_s^H
	\end{equation}
where the processes $(\tilde{\vartheta})$, $(\tilde{\varphi})$, and $(\tilde{\psi})$ are given in \eqref{eq:Ito_formula_full}. Only the following convergence is proved:
	\begin{equation}
	\label{eq:Ito_2}
		\frac{1}{\varepsilon}\int_0^t\int_s^{s+\varepsilon}\frac{\partial f}{\partial x}(s,x_s)\psi_u\delta R_u^H\,\mathrm{d}{s}\quad \overset{L^2(\Omega)}{\underset{\varepsilon\downarrow 0}{\longrightarrow }}\quad\int_0^t \frac{\partial f}{\partial x}(s,x_s)\psi_s\delta R_s^H.
	\end{equation}
The convergence of the other terms can be shown in a similar manner. By using the Fubini-type theorem in \autoref{lem:fubini}, it follows that
	\begin{equation*}
		\frac{1}{\varepsilon}\int_0^t\int_s^{s+\varepsilon}\frac{\partial f}{\partial x}(s,x_s)\psi_u\delta R_u^H\,\mathrm{d}{s} = \frac{1}{\varepsilon}\int_0^{t+\varepsilon}\left(\int_{(u-\varepsilon)\vee 0}^{u\wedge t}\frac{\partial f}{\partial x}(s,x_s)\psi_u\,\mathrm{d}{s}\right)\delta R_u^H
	\end{equation*}
is satisfied almost surely. By using \autoref{prop:boundedness_of_int_R} and \cite[Proposition 1.5.6]{Nua06}, it follows that
	\begin{align*}
		 \left\|\frac{1}{\varepsilon}\int_\varepsilon^{t}\int_{u-\varepsilon}^u \frac{\partial f}{\partial x}(s,x_s)\psi_u\,\mathrm{d}{s}\delta R_u^H - \int_\varepsilon^t \frac{\partial f}{\partial x}(u,x_u)\psi_u\delta R_u^H\right\|_{L^2(\Omega)}^{H} & \lesssim\\
		 & \hspace{-4cm} \lesssim \int_\varepsilon^t\left\|\psi_u\left[\frac{1}{\varepsilon}\int_{u-\varepsilon}^u\frac{\partial f}{\partial x}(s,x_s)\,\mathrm{d}{s} - \frac{\partial f}{\partial x}(u,x_u)\right]\right\|_{\mathbb{D}^{2,2}}^\frac{1}{H}\,\mathrm{d}{u} \\
		 &\hspace{-4cm} \lesssim \int_\varepsilon^t\|\psi_u\|_{\mathbb{D}^{2,4}}^\frac{1}{H}\left\|\frac{1}{\varepsilon}\int_{u-\varepsilon}^u\frac{\partial f}{\partial x}(s,x_s)\,\mathrm{d}{s} - \frac{\partial f}{\partial x}(u,x_u)\right\|_{\mathbb{D}^{2,4}}^\frac{1}{H}\,\mathrm{d}{u}.
	\end{align*}
The convergence
	\begin{equation*}
		\frac{1}{\varepsilon}\int_{u-\varepsilon}^u\frac{\partial f}{\partial x}(s,x_s)\,\mathrm{d}{s} \quad\overset{\mathbb{D}^{2,4}}{\underset{\varepsilon\downarrow 0}{\longrightarrow}} \quad \frac{\partial f}{\partial x}(u,x_u)
	\end{equation*}
is satisfied for almost every $u\in [0,t]$ by the Lebesgue differentiation theorem because
	\begin{equation*}
		\int_0^t\left\|\frac{\partial f}{\partial x}(s,x_s)\right\|_{\mathbb{D}^{2,4}}\,\mathrm{d}{s}\leq t\,\esssup_{u\in (0,t)}\left\|\frac{\partial f}{\partial x}(s,x_s)\right\|_{\mathbb{D}^{2,4}} <\infty.
	\end{equation*}
Moreover, the estimate
	\begin{equation*}
	\|\psi_u\|_{\mathbb{D}^{2,4}}^\frac{1}{H}\left\|\frac{1}{\varepsilon}\int_{u-\varepsilon}^u\left[\frac{\partial f}{\partial x}(s,x_s) - \frac{\partial f}{\partial x}(u,x_u)\right]\,\mathrm{d}{s}\right\|_{\mathbb{D}^{2,4}}^\frac{1}{H}\lesssim \|\psi_u\|_{\mathbb{D}^{2,4}}^\frac{1}{H}\left(\esssup_{s\in (0,t)}\left\|\frac{\partial f}{\partial x}(s,x_s)\right\|_{\mathbb{D}^{2,4}}\right)^\frac{1}{H}
	\end{equation*}
is satisfied for sufficiently small $\varepsilon>0$, and because the right-hand side is integrable (with respect to $\,\mathrm{d}{u}$) on the interval $(0,t)$, Lebesgue's dominated convergence theorem yields the desired convergence \eqref{eq:Ito_2}.
\end{proof}

\begin{remark}
The significance of \autoref{prop:Ito_formula} is two-fold. Firstly, it describes the general structure of the It\^o-type formula that should be expected when a Rosenblatt integrator is involved. Secondly, it gives a general method of proof that could be employed in concrete situations. Moreover, it is expected that the equality \eqref{eq:Ito_formula_full} will be satisfied even under a different set of assumptions that are more suitable in specific cases. Some remarks of the assumptions used here follow.
	\begin{enumerate}[label=(\roman*)]
		\itemsep0em
		\item The assumption of H\"older continuity of $x$ is rather natural since it is known that the Skorokhod-type integrals retain H\"older continuity of the integrator (under suitable conditions on the integrand), see \cite[Theorem 5]{AlosNua03} for fractional integrators and \cite[Proposition 4]{Tud08} for Rosenblatt integrators.
		\item Polynomial growth of $f$ as well as the corresponding integrability of the processes $(\vartheta)$, $(\varphi)$, and $(\psi)$ is assumed for the purposes of the chain rule for Malliavin derivative in \autoref{cor:chain_rule_H}. However, for a specific  problem, a more suitable version of the chain rule can be used, see e.g. \cite[Proposition 1.2.3 or Proposition 1.2.4]{Nua06}. This would lead to a different set of assumptions.
		\item It is not assumed that any of the processes $(\vartheta)$, $(\varphi)$, or $(\psi)$ are adapted.
		\item Assumptions \ref{ass:gen_4} - \ref{ass:gen_7} are not independent of assumptions \ref{ass:gen_1} - \ref{ass:gen_3} and they can pose even stronger conditions on the processes $(\vartheta)$, $(\varphi)$, and $(\psi)$. However, in many cases, it can be convenient to verify conditions that are formulated in terms of the process $x$. In \autoref{prop:Ito_for_integral_only}, it is shown how to obtain conditions solely in terms of the integrand.
	\end{enumerate}
\end{remark}

\begin{remark}
Note that the convergence of the terms with the remainders $R_1, R_2$, and $R_3$ in the equality \eqref{eq:Ito_4} shows that the forward integral
			\begin{equation*}
				\int_0^t\frac{\partial f}{\partial x}(s,x_s)\,\mathrm{d}^{-}x_s
			\end{equation*}
		exists and that
			\begin{equation}
			\label{eq:Ito_forward}
				f(t,x_t)-f(0,x_0) = \int_0^t\frac{\partial f}{\partial s}(s,x_s)\,\mathrm{d}{s} + \int_0^t\frac{\partial f}{\partial x}(s,x_s)\,\mathrm{d}^{-}x_s
			\end{equation}
		is satisfied almost surely (this is, in fact, the statement of \cite[Theorem 1.2]{RusVal95} for processes with zero quadratic variation). Therefore, if the equality
			\begin{align*}
				\int_0^t\frac{\partial f}{\partial x}(s,x_s)\,\mathrm{d}^{-}x_s & = \int_0^t \frac{\partial f}{\partial x}(s,x_s)\,\mathrm{d}^{-}R_s^H + 2c_H^{B,R}\int_0^t \frac{\partial f}{\partial x}(s,x_s)[\varphi_s-(\nabla^\frac{H}{2}\psi_s)(s)]\,\mathrm{d}^{-}B_s^{\frac{H}{2}+\frac{1}{2}} \\
				& \hspace{1cm} + \int_0^t \frac{\partial f}{\partial x}(s,x_s)\bigg[\vartheta_s  - 2c_H^R(\nabla^\frac{H}{2}\varphi_s)(s) + c_H^R (\nabla^{\frac{H}{2},\frac{H}{2}}\psi_s)(s,s)\bigg]\,\mathrm{d}{s}, \numberthis\label{eq:forward_transition}
			\end{align*}
is satisfied almost surely, the proof of \autoref{prop:Ito_formula} could be as follows: Rewrite the Skorokhod integrals in \eqref{eq:fractional_differential} as forward integrals by means of \autoref{prop:relationship_for_Skor_B} and \autoref{prop:relationship_for_Skor_R}, use the It\^o formula for processes with forward differential \eqref{eq:Ito_forward} and equality \eqref{eq:forward_transition}, and rewrite the resulting forward integrals back in their Skorokhod form. To be able to employ this strategy, it would be desirable to find general conditions under which the equality
			\begin{equation}
			\label{eq:interchange_formula}
				\int_0^T g_s^{(1)}\,\mathrm{d}^{-}H_s = \int_0^T g_s^{(1)}g_s^{(2)}\,\mathrm{d}^{-}h_s
			\end{equation}
is satisfied almost surely for some continuous process $(g_t^{(1)})_{t\in[0,T]}$ and the process $(H_t)_{t\in [0,T]}$ that is given by
	\begin{equation*}
		H_t = \int_0^t g_s^{(2)}\,\mathrm{d}^{-}h_s
	\end{equation*}
with some integrable process $(g_t^{(2)})_{t\in [0,T]}$ and continuous process $(h_t)_{t\in [0,T]}$. This problem is also the subject of \cite[Remark 5.9]{Zahle99}. Formally, there are the following equalities
	\begin{align*}
		\int_0^t g_s^{(1)}\,\mathrm{d}^{-}H_s & = \plim_{\varepsilon\downarrow 0} \frac{1}{\varepsilon}\int_0^t g_s^{(1)}\int_s^{s+\varepsilon}g_u^{(2)}\,\mathrm{d}^{-}H_u\,\mathrm{d}{s}\\
		& =\plim_{\varepsilon\downarrow 0}\frac{1}{\varepsilon}\int_0^tg_s^{(1)}\left(\plim_{\delta\downarrow 0}\frac{1}{\delta }\int_s^{s+\varepsilon}g_u^{(2)}(h_{u+\delta}-h_u)\,\mathrm{d}{u}\right)\,\mathrm{d}{s}\numberthis\label{eq:interchange_1}\\
		& = \plim_{\varepsilon\downarrow 0}\plim_{\delta\downarrow 0}\frac{1}{\varepsilon}\int_0^tg_s^{(1)}\left(\frac{1}{\delta }\int_s^{s+\varepsilon}g_u^{(2)}(h_{u+\delta}-h_u)\,\mathrm{d}{u}\right)\,\mathrm{d}{s}\numberthis\label{eq:interchange_2}\\
		& = \plim_{\delta\downarrow 0}\plim_{\varepsilon\downarrow 0}\frac{1}{\delta}\int_0^tg_s^{(1)}\left(\frac{1}{\varepsilon }\int_s^{s+\varepsilon}g_u^{(2)}(h_{u+\delta}-h_u)\,\mathrm{d}{u}\right)\,\mathrm{d}{s}\numberthis\label{eq:interchange_3}\\
		& = \plim_{\delta\downarrow 0}\frac{1}{\delta}\int_0^tg_s^{(1)}\left(\plim_{\varepsilon\downarrow 0}\frac{1}{\varepsilon }\int_s^{s+\varepsilon}g_u^{(2)}(h_{u+\delta}-h_u)\,\mathrm{d}{u}\right)\,\mathrm{d}{s}\\
		& = \plim_{\delta\downarrow 0}\frac{1}{\delta}\int_0^tg_s^{(1)}g_s^{(2)}(h_{s+\varepsilon}+h_{s})\,\mathrm{d}{s}\\
		& = \int_0^tg_s^{(1)}g_s^{(2)}\,\mathrm{d}^{-}h_s.
	\end{align*}
The above formal computation would be correct provided that the probability limit can be interchanged with the outer integral in \eqref{eq:interchange_1} and \eqref{eq:interchange_3}, and that the interchange of the probability limits in \eqref{eq:interchange_2} is also possible. Some criteria for the interchange of probability limits and the integrals can be found using Markov's inequality and a Moore-Osgood-type argument could allow to interchange the two probability limits. In the case of \autoref{prop:Ito_formula}, however, the validity of \eqref{eq:forward_transition} is shown at the level of $\varepsilon$-approximations via the $L^2(\Omega)$-convergence \eqref{eq:Ito_3}.
\end{remark}

\begin{remark}
\autoref{prop:Ito_formula} is compared to the It\^o formula for regular fractional Brownian motions in \cite[Theorem 4.5]{DunHuDun00}. In this remark, the same symbols as in \autoref{prop:Ito_formula} are used to complete the analogy, however note that the objects can be different. Thereom 4.5 of \cite{DunHuDun00} states that if $(x_t)$ is the stochastic process defined by
	\begin{equation*}
		x_t = x_0 + \int_0^t\vartheta_s\,\mathrm{d}{s} + \int_0^t\varphi_s\delta B^{H}_s,
	\end{equation*}
where $H\in (\sfrac{1}{2},1)$ and $(\vartheta)$ and $(\varphi)$ are stochastic process that satisfy suitable integrability assumptions, then the equality
	\begin{equation}
	\label{eq:Ito_fbm}
		f(t,x_t) = f(0,x_0) + \int_0^t\left[\frac{\partial f}{\partial s}(s,x_s) + \frac{\partial f}{\partial x}(s,x_s)\vartheta_s + \frac{\partial f}{\partial x}(s,x_s)\varphi_s D^\phi_sx_s\right]\,\mathrm{d}{s} + \int_0^t\frac{\partial f}{\partial x}(s,x_s)\varphi_s\delta B_s^{H}.
	\end{equation}
is satisfied almost surely. In formula \eqref{eq:Ito_fbm}, the operator $D^\phi$ is the operator $\nabla^{H-\frac{1}{2}}$ (up to a constant). Therefore, the structure of a process with a first-order fractional differential is preserved under compositions with $\mathscr{C}^2$ functions; that is,  for $y_t=f(t,x_t)$ it follows that
	\begin{equation*}
		y_t=y_0+ \int_0^t \tilde{\vartheta}_s\,\mathrm{d}{s} + \int_0^t \tilde{\varphi}_s\delta B_s^{H}.
	\end{equation*}
In the case where $(x_t)$ is defined by
	\begin{equation*}
		x_t = x_0 + \int_0^t\vartheta_s\,\mathrm{d}{s} + 2c_H^{B,R}\int_0^t \varphi_s\delta B_s^{\frac{H}{2}+\frac{1}{2}} + \int_0^t\psi_s\delta R_s^H,
	\end{equation*}
the situation is analogous to the case of fractional Brownian motions because the following formula for $y_t=f(t,x_t)$ is obtained:
	\begin{equation*}
		y_t = y_0  + \int_0^t\tilde{\vartheta}_s\,\mathrm{d}{s} + 2c_H^{B,R}\int_0^t \tilde{\varphi}_s\delta B_s^{\frac{H}{2}+\frac{1}{2}} + \int_0^t\tilde{\psi}_s\delta R_s^H.
	\end{equation*}
On the other hand, a notable difference between formula \eqref{eq:Ito_fbm} and formula \eqref{eq:Ito_formula_full} is the appearance of a term that involves the third derivative $\frac{\partial^3f}{\partial x^3}(s,x_s)$ and a term that involves the second-order fractional stochastic derivative $\nabla^{\frac{H}{2},\frac{H}{2}}x_s$. Both of these terms arise as a result of the chain rule used to compute the second-order fractional stochastic derivative $\nabla^{\frac{H}{2},\frac{H}{2}}(\frac{\partial f}{\partial x}(s,x_s))$. Thus, the appearance of these new terms is a direct consequence of the second-order nature of Rosenblatt processes. These phenomena are also discussed in \cite[p. 548]{Arr15} and in \cite[Remark 8]{Tud08}.
\end{remark}

For applications of \autoref{prop:Ito_formula}, it is necessary to compute explicit formulas for $\nabla^\frac{H}{2}x_t$ and $\nabla^{\frac{H}{2},\frac{H}{2}}x_t$ where $x_t$ is given by
	\begin{equation*}
		x_t = x_0 + \int_0^t\vartheta_s\,\mathrm{d}{s} + 2c_H^{B,R}\int_0^t\varphi_s\delta B_s^{\frac{H}{2}+\frac{1}{2}} + \int_0^t\psi_s\delta R_s^H.
	\end{equation*}
This requires being able to compute both the first and second-order fractional stochastic derivative of the Skorokhod integral with respect to both the fractional Brownian motion and the Rosenblatt process.  Explicit formulas are given in the following four lemmas. These lemmas are proved by, possibly iterative, use of \cite[Proposition 1.3.8]{Nua06} and formula \eqref{eq:integral_beta}.

The first two lemmas give expressions for the first and second-order fractional stochastic derivatives of the Skorokhod integral with respect to a fractional Brownian motion.

\begin{lemma}
\label{prop:1der1int}
Let $g\in L^\frac{2}{1+H}(0,T;\mathbb{D}^{2,2})$. Then the following equality is satisfied for almost every $x\in\mathbb{R}$:
	\begin{equation*}
		\nabla^\frac{H}{2}\left(\int_0^Tg_s\delta B_s^{\frac{H}{2}+\frac{1}{2}}\right)(x) = \int_0^T(\nabla^\frac{H}{2}g_s)(x)\delta B_s^{\frac{H}{2}+\frac{1}{2}}  + c_{\frac{H}{2}+\frac{1}{2}}^B\frac{\mathrm{B}\left(\frac{H}{2},1-H\right)}{\Gamma\left(\frac{H}{2}\right)^2}\int_0^Tg_s|s-x|^{H-1}\,\mathrm{d}{s}.
	\end{equation*}
\end{lemma}

The above \autoref{prop:1der1int} should be compared to \cite[Theorem 4.2]{DunHuDun00}. Note, in particular, that the constant appearing in front of the second integral is different.

\begin{lemma}
\label{prop:2der_1int}
Let $g\in L^\frac{2}{1+H}(0,T;\mathbb{D}^{3,2})$. Then the following equality is satisfied for almost every $x,y\in\mathbb{R}$:
	\begin{align*}
		\nabla^{\frac{H}{2},\frac{H}{2}}\left(\int_0^Tg_s\delta B_s^{\frac{H}{2}+\frac{1}{2}}\right)(x,y) & = \int_0^T(\nabla^{\frac{H}{2},\frac{H}{2}}g_s)(x,y)\delta B_s^{\frac{H}{2}+\frac{1}{2}}\\
		& \hspace{-3cm} + c_{\frac{H}{2}+\frac{1}{2}}^B\frac{\mathrm{B}\left(\frac{H}{2},1-H\right)}{\Gamma\left(\frac{H}{2}\right)^2} \left(\int_0^T(\nabla^\frac{H}{2}g_s)(x)|s-y|^{H-1}\,\mathrm{d}{s} + \int_0^T(\nabla^\frac{H}{2}g_s)(y)|s-x|^{H-1}\,\mathrm{d}{s}\right).
	\end{align*}
\end{lemma}

The following two propositions can be used to compute the first and second-order fractional stochastic derivative of the Skorokhod integral with respect to a Rosenblatt process.

\begin{lemma}
\label{prop:1der_2int}
Let $g\in L^\frac{1}{H}(0,T;\mathbb{D}^{3,2})$. Then the following equality is satisfied for almost every $x\in\mathbb{R}$:
\begin{align*}
	\nabla^{\frac{H}{2}}\left(\int_0^Tg_s\delta R_s^H\right)(x) & = \int_0^T(\nabla^\frac{H}{2}g_s)(x)\delta R_s^H + 2c_{H}^{B,R}\frac{\mathrm{B}\left(\frac{H}{2},1-H\right)}{\Gamma\left(\frac{H}{2}\right)^2}\int_0^Tg_s|s-x|^{H-1}\delta B_s^{\frac{H}{2}+\frac{1}{2}}.
\end{align*}
\end{lemma}

\begin{lemma} Let $g\in L^\frac{1}{H}(0,T;\mathbb{D}^{4,2})$. Then the following equality is satisfied for almost every $x,y\in\mathbb{R}$:
\label{prop:2der_2int}
\begin{align*}
	\nabla^{\frac{H}{2},\frac{H}{2}}\left(\int_0^Tg_s\delta R_s^H\right)(x,y) & = \int_0^T(\nabla^{\frac{H}{2},\frac{H}{2}}g_s)(x,y)\delta R_s^H\\
		& \hspace{-3.1cm} + 2c_H^{B,R}\frac{\mathrm{B}\left(\frac{H}{2},1-H\right)}{\Gamma\left(\frac{H}{2}\right)^2}\left(\int_0^T(\nabla^\frac{H}{2}g_s)(x)|s-y|^{H-1}\delta{B}_s^{\frac{H}{2}+\frac{1}{2}} + \int_0^T(\nabla^\frac{H}{2}g_s)(y)|s-x|^{H-1}\delta{B}_s^{\frac{H}{2}+\frac{1}{2}}\right) \\
		& \hspace{-3.1cm} + 2c_H^R\frac{\mathrm{B}\left(\frac{H}{2},1-H\right)^2}{\Gamma\left(\frac{H}{2}\right)^4}\int_0^Tg_s|s-x|^{H-1}|s-y|^{H-1}\,\mathrm{d}{s}.
\end{align*}
\end{lemma}

Some special cases of \autoref{prop:Ito_formula} are given now. Since the It\^o formula for functionals of Skorokhod integrals with respect to the fractional Brownian motion is well-known, e.g. \cite{DunHuDun00}, functionals of the Skorokhod integral with respect to a Rosenblatt process are considered. Moreover, in this case, it is possible to formulate sufficient conditions for the It\^o formula in terms of the integrand rather than in terms of the integral.

\begin{proposition}
\label{prop:Ito_for_integral_only}
Let $f$ be a function in $\mathscr{C}^3(\mathbb{R})$ such that
	\begin{equation}
	\label{eq:poly_growth}
		|f'''(x)|\leq c(1+|x|^\alpha), \quad x\in\mathbb{R},
	\end{equation}
is satisfied for some $c\geq 0$ and $\alpha\geq 0$. Let $\psi=(\psi_s)_{s\in [0,T]}$ be a stochastic process such that the following three assumptions are satisfied:

\begin{enumerate}[label=\upshape{(Z\arabic*)}]
	\item\label{ass:psi_1} The proces $\psi$ is in $L^\infty(0,T;\mathbb{D}^{4,p})$ for some
		\begin{equation*}	
			p>\max\left\{\frac{2}{(2H-1)}, 8(\alpha+1)\right\}.
		\end{equation*}
	\item\label{ass:psi_4} For almost every $r,v\in [0,T]$, the following equalities are satisfied
	\begin{align*}
		& \lim_{\varepsilon\downarrow 0}\esssup_{u\in (v,v+\varepsilon)}\|(\nabla^\frac{H}{2}\psi_r)(u)-(\nabla^\frac{H}{2}\psi_r)(v)\|_{\mathbb{D}^{3,8}} = 0,\\
		&\lim_{\varepsilon\downarrow 0}\esssup_{u\in (v,v+\varepsilon)} \|(\nabla^{\frac{H}{2}, \frac{H}{2}}\psi_r)(u,u)-(\nabla^{\frac{H}{2},\frac{H}{2}}\psi_r)(v,v)\|_{\mathbb{D}^{2,8}}=0, \\
		&\lim_{\varepsilon\downarrow 0}\esssup_{u\in (v,v+\varepsilon)} \|(\nabla^{\frac{H}{2}, \frac{H}{2}}\psi_r)(r,u)-(\nabla^{\frac{H}{2},\frac{H}{2}}\psi_r)(r,v)\|_{\mathbb{D}^{2,8}}=0.
	\end{align*}
	\item\label{ass:psi_2} There exist functions $p_1\in L^\frac{6}{1+H}(0,T)$, $p_2\in L^\frac{1}{H}(0,T)$, and $p_3\in L^\frac{1}{H}(0,T)$ such that for almost every $r\in [0,T]$, the estimates
		\begin{align*}
			& \|(\nabla^\frac{H}{2}\psi_r)(v)\|_{\mathbb{D}^{3,8}}\leq p_1(r)\\
			& \|(\nabla^{\frac{H}{2},\frac{H}{2}}\psi_r)(v,v)\|_{\mathbb{D}^{2,8}}\leq p_2(r)\\
			& \|(\nabla^{\frac{H}{2},\frac{H}{2}}\psi_r)(r,v)\|_{\mathbb{D}^{2,8}}\leq p_3(r)
		\end{align*}
	are satisfied for almost every $v\in [0,T]$.
\end{enumerate}
Define the process $(Z_t)_{t\in [0,T]}$ by
	\begin{equation}
	\label{eq:Z}
		Z_t \overset{\mathrm{Def.}}{=}\int_0^t\psi_r\delta R_r^H.
	\end{equation}
Then the following equality is satisfied for every $t\in [0,T]$ almost surely:
\begin{align*}
		f(Z_t) - f(0) & = \int_0^t f'(Z_s)\psi_s\delta R_s^H \\
		& \hspace{1cm} + 2c_H^{B,R}\int_0^tf''(Z_s)(\nabla^\frac{H}{2}Z_s)(s)\psi_s\delta B_s^{\frac{H}{2}+\frac{1}{2}} \\
		& \hspace{2cm} + c_H^R\int_0^t \left(f''(Z_s)(\nabla^{\frac{H}{2},\frac{H}{2}}Z_s)(s,s) + f'''(Z_s)[(\nabla^\frac{H}{2}Z_s)(s)]^2\right)\psi_s\,\mathrm{d}{s}.\numberthis\label{eq:Ito_for_integral}
	\end{align*}
\end{proposition}

\begin{proof}
The conditions of \autoref{prop:Ito_formula} are verified. Condition \ref{ass:gen_3} of \autoref{prop:Ito_formula} is clearly satisfied. It is therefore necessary to verify that the process $Z$ has a version with H\"older continuous sample paths of order greater than $\sfrac{1}{2}$ (that is considered in the sequel) and that conditions \ref{ass:gen_4} - \ref{ass:gen_7} of \autoref{prop:Ito_formula} are satisfied.

\medskip

\textit{Claim 1:} The process $Z$ has a version with H\"older continuous sample paths of an order greater than $\sfrac{1}{2}$.

\textit{Proof of Claim 1:} By using \autoref{prop:boundedness_of_int_R}, the embedding $\mathbb{D}^{4,p}\hookrightarrow \mathbb{D}^{2,p}$, and assumption \ref{ass:psi_1} successively, it follows that
	\begin{equation*}
		\mathbb{E}|Z_t-Z_s|^p=\mathbb{E}\left|\int_s^t \psi_r\delta R_r^H\right|^p \lesssim \left(\int_s^t \|\psi_r\|_{\mathbb{D}^{2,p}}^\frac{1}{H}\,\mathrm{d}{r}\right)^{pH} \leq \|\psi\|_{L^\infty(0,T;\mathbb{D}^{4,p})}^p (t-s)^{pH}
	\end{equation*}
is satisfied for $0\leq s<t\leq T$. This implies, by the Kolmogorov-Chentsov criterion, that the process $(Z_t)_{t\in [0,T]}$ has a version with H\"older continuous sample paths of every order less than $H-\sfrac{1}{p}$. Clearly,  $H-\sfrac{1}{p} > \sfrac{1}{2}$, see also \cite[Proposition 4]{Tud08}.

\medskip

\textit{Claim 2:} The process $Z$ satisfies condition \ref{ass:gen_4} of \autoref{prop:Ito_formula}.

\textit{Proof of Claim 2:} By using the estimate \eqref{eq:estimates_on_norms_of_f}, and \autoref{prop:boundedness_of_int_R} the estimate
	\begin{align*}
		\esssup_{s\in[0,T]}\|f'(Z_s)\|_{\mathbb{D}^{2,4}} & \lesssim 1 + \esssup_{s\in [0,T]}\|Z_s\|_{\mathbb{D}^{2,4(\alpha+2)}}^{\alpha+2}\\
		& \lesssim 1 + \esssup_{s\in [0,T]}\left(\int_0^s\|\psi_r\|_{\mathbb{D}^{4,4(\alpha+2)}}^\frac{1}{H}\,\mathrm{d}{r}\right)^{H(\alpha+2)}\\
		& \lesssim T^{H(\alpha+2)}\|\psi\|_{L^\infty(0,T;\mathbb{D}^{4,4(\alpha+2)})}^{\alpha+2}.
	\end{align*}
is obtained. The last expression is finite by the embedding $\mathbb{D}^{4,p}\hookrightarrow \mathbb{D}^{4(\alpha+2)}$ and assumption \ref{ass:psi_1}. Similarly, it follows that
	\begin{equation*}
		\esssup_{s\in[0,T]}\|f''(Z_s)\|_{\mathbb{D}^{1,8}} \lesssim 1 + \esssup_{s\in [0,T]}\|Z_s\|^{\alpha+1}_{\mathbb{D}^{1,8(\alpha +1)}} \lesssim 1 + \left(\int_0^T\|\psi_r\|^\frac{1}{H}_{\mathbb{D}^{1,8(\alpha+1)}}\,\mathrm{d}{r}\right)^{H(\alpha+1)} <\infty
	\end{equation*}
and also $\esssup_{s\in[0,T]} \|f'''(Z_s)\|_{L^8(\Omega)} <\infty$.

\medskip

\textit{Claim 3:} The process $Z$ satisfies condition \ref{ass:gen_5} of \autoref{prop:Ito_formula}.

\textit{Proof of Claim 3:} By \autoref{prop:1der_2int}, the inequality
	\begin{equation*}
		\|(\nabla^\frac{H}{2}Z_u)(v)-(\nabla^\frac{H}{2}Z_v)(v)\|_{\mathbb{D}^{1,8}} \lesssim \left\|\int_u^v (\nabla^\frac{H}{2}\psi_r)(v)\delta R_r^H\right\|_{\mathbb{D}^{1,8}} + \left\|\int_u^v\psi_r|v-r|^{H-1}\delta{B}_r^{\frac{H}{2}+\frac{1}{2}}\right\|_{\mathbb{D}^{1,8}}
	\end{equation*}
is satisfied for almost every $u,v\in [0,T]$ such that $0<u<v<T$. By using \autoref{prop:boundedness_of_int_R} and assumption \ref{ass:psi_2}, it follows for the norm of the first integral that
	\begin{align*}
		\lim_{\varepsilon\downarrow 0} \esssup_{u\in (v-\varepsilon,v)}\left\|\int_u^v (\nabla^\frac{H}{2}\psi_r)(v)\delta R_r^H\right\|_{\mathbb{D}^{1,8}} & \lesssim\lim_{\varepsilon\downarrow 0} \esssup_{u\in (v-\varepsilon,v)} \left(\int_u^v\|(\nabla^\frac{H}{2}\psi_r)(v)\|_{\mathbb{D}^{3,8}}^\frac{1}{H}\,\mathrm{d}{r}\right)^H \\
		& \leq \lim_{\varepsilon\downarrow 0} \esssup_{u\in (v-\varepsilon,v)} \left(\int_{u}^vp_1(r)^\frac{1}{H}\,\mathrm{d}{r}\right)^H \\
		& = \lim_{\varepsilon\downarrow 0} \left(\int_{v-\varepsilon}^vp_1(r)^\frac{1}{H}\,\mathrm{d}{r}\right)^H
	\end{align*}
for almost every $v\in [0,T]$. The last limit is zero since $p_1$ belongs to the space $L^\frac{6}{H+1}(0,T)$ by \ref{ass:psi_2}. Similarly for the norm of the second integral, it follows that
\begin{align*}
	\lim_{\varepsilon\downarrow 0}\esssup_{u\in (v-\varepsilon,v)} \left\|\int_u^v\psi_r|v-r|^{H-1}\delta{B}_r^{\frac{H}{2}+\frac{1}{2}}\right\|_{\mathbb{D}^{1,8}} & \lesssim \lim_{\varepsilon\downarrow 0}\esssup_{u\in (v-\varepsilon,v)} \left(\int_u^v\|\psi_r\|_{\mathbb{D}^{2,8}}^\frac{2}{H+1}(v-r)^\frac{2(H-1)}{H+1}\,\mathrm{d}{r}\right)^\frac{H+1}{2}\\
	& \lesssim \|\psi\|_{L^\infty(0,T;\mathbb{D}^{2,8})}\lim_{\varepsilon\downarrow 0}\left(\int_{v-\varepsilon}^v(v-r)^\frac{2(H-1)}{H+1}\,\mathrm{d}{r}\right)^\frac{H+1}{2}
\end{align*}
is satisfied for almost every $v\in [0,T]$ by using \autoref{prop:boundedness_of_int_B}. The last expression is clearly zero since the norm is finite by the embedding $\mathbb{D}^{4,p}\hookrightarrow \mathbb{D}^{2,8}$ and assumption \ref{ass:psi_1}. Similarly, by using \autoref{prop:2der_2int} and \autoref{prop:boundedness_of_int_B} and \autoref{prop:boundedness_of_int_R} successively, the inequality
\begin{align*}
	\|(\nabla^{\frac{H}{2},\frac{H}{2}}Z_u)(v,v)-(\nabla^{\frac{H}{2},\frac{H}{2}}Z_v)(v,v)\|_{L^8(\Omega)} & \lesssim \left(\int_u^v\|(\nabla^{\frac{H}{2},\frac{H}{2}}\psi_r)(v,v)\|_{\mathbb{D}^{2,8}}^\frac{1}{H}\,\mathrm{d}{r}\right)^H \\
	& \hspace{1cm} + \left(\int_u^v\|(\nabla^\frac{H}{2}\psi_r)(v)\|_{\mathbb{D}^{1,8}}^\frac{2}{H+1}|v-r|^\frac{2(H-1)}{H+1}\,\mathrm{d}{r}\right)^\frac{H+1}{2} \\
	& \hspace{1cm} + \int_u^v\|\psi_r\|_{L^8(\Omega)}|v-r|^{2H-2}\,\mathrm{d}{r}\\
	& \leq \left(\int_u^vp_2(r)^\frac{1}{H}\,\mathrm{d}{r}\right)^H \\
	& \hspace{1cm} + \left(\int_u^vp_1(r)^\frac{6}{H+1}\,\mathrm{d}{r}\right)^\frac{H+1}{6}\left(\int_u^v|v-r|^\frac{3(H-1)}{H+1}\,\mathrm{d}{r}\right)^\frac{H+1}{3(H-1)} \\
	& \hspace{1cm} + \|\psi\|_{L^\infty(0,T;L^8(\Omega))}\int_u^v|v-r|^{2H-2}\,\mathrm{d}{r}
\end{align*}
is satisfied for almost every $u,v\in [0,T]$ such that $v>u$. The second inequality is obtained by using assumption \ref{ass:psi_2} and H\"older's inequality. By the embedding $\mathbb{D}^{4,p}\hookrightarrow L^8(\Omega)$ and assumptions \ref{ass:psi_1} and \ref{ass:psi_2}, it follows from the last inequality that
	\begin{equation*}
		\lim_{\varepsilon\downarrow 0}\esssup_{u\in (v-\varepsilon,v)}\|(\nabla^{\frac{H}{2},\frac{H}{2}}Z_u)(v,v)-(\nabla^{\frac{H}{2},\frac{H}{2}}Z_v)(v,v)\|_{L^8(\Omega)}=0
	\end{equation*}
is satisfied for almost every $v\in [0,T]$.

\medskip

\textit{Claim 4:} The process $Z$ satisfies condition \ref{ass:gen_6} of \autoref{prop:Ito_formula}.

\textit{Proof of Claim 4:} By using \autoref{prop:1der_2int}, it follows that the estimate
	\begin{align*}
		\|(\nabla^\frac{H}{2}Z_s)(u)-(\nabla^\frac{H}{2}Z_s)(v)\|_{\mathbb{D}^{1,8}} & \lesssim \left\|\int_0^s\left[(\nabla^\frac{H}{2}\psi_r)(u)-(\nabla^\frac{H}{2}\psi_r)(v)\right]\delta{R}_r^H\right\|_{\mathbb{D}^{1,8}} \\
		& \hspace{1cm} + \left\|\int_0^s\psi_r\left[|u-r|^{H-1}-|v-r|^{H-1}\right]\delta B_r^{\frac{H}{2}+\frac{1}{2}}\right\|_{\mathbb{D}^{1,8}}\numberthis\label{eq:claim_4_1}
	\end{align*}
is satisfied for every $s\in [0,T]$ and almost every $u,v\in [0,T]$ such that $u>v$. By using \autoref{prop:boundedness_of_int_R}, the following estimate for the norm of the first stochastic integral in \eqref{eq:claim_4_1} is obtained:
	\begin{equation*}
		\left\|\int_0^s\left[(\nabla^\frac{H}{2}\psi_r)(u)-(\nabla^\frac{H}{2}\psi_r)(v)\right]\delta{R}_r^H\right\|_{\mathbb{D}^{1,8}} \lesssim \left(\int_0^s\|(\nabla^\frac{H}{2}\psi_r)(u)-(\nabla^\frac{H}{2}\psi_r)(v)\|_{\mathbb{D}^{3,8}}^\frac{1}{H}\,\mathrm{d}{r}\right)^H.
	\end{equation*}
The integrand in the last expression tends to zero as $u\downarrow v$ by assumption \ref{ass:psi_4} and it is dominated by $p_1$ by assumption \ref{ass:psi_2}. Consequently, it follows that
	\begin{equation*}
		\lim_{\varepsilon\downarrow 0}\esssup_{u\in (v,v+\varepsilon)} \left\|\int_0^s\left[(\nabla^\frac{H}{2}\psi_r)(u)-(\nabla^\frac{H}{2}\psi_r)(v)\right]\delta{R}_r^H\right\|_{\mathbb{D}^{1,8}} = 0
	\end{equation*}
is satisfied for every $s\in [0,T]$ by the Lebesgue's dominated convergence theorem. Similarly, by \autoref{prop:boundedness_of_int_B}, the following estimate for the norm of the second integral in \eqref{eq:claim_4_1} is obtained:
	\begin{equation*}
		\left\|\int_0^s\psi_r\left[|u-r|^{H-1}-|v-r|^{H-1}\right]\delta B_r^{\frac{H}{2}+\frac{1}{2}}\right\|_{\mathbb{D}^{1,8}} \lesssim \left(\int_0^s\|\psi_r\|_{\mathbb{D}^{2,8}}^\frac{2}{H+1}\left||u-r|^{H-1}-|v-r|^{H-1}\right|^{\frac{2}{H+1}}\,\mathrm{d}{r}\right)^\frac{H+1}{2}.
	\end{equation*}
This estimate yields, by using the embedding $\mathbb{D}^{4,p}\hookrightarrow \mathbb{D}^{2,8}$ and assumption \eqref{ass:psi_1}, that
	\begin{equation*}
		\lim_{\varepsilon\downarrow 0}\esssup_{u\in (v,v+\varepsilon)} \left\|\int_0^s\psi_r\left[|u-r|^{H-1}-|v-r|^{H-1}\right]\delta B_r^{\frac{H}{2}+\frac{1}{2}}\right\|_{\mathbb{D}^{1,8}} =0
	\end{equation*}
is satisfied for every $s\in [0,T]$. By similar arguments, it can also be shown that
	\begin{equation*}
		\lim_{\varepsilon\downarrow 0}\esssup_{u\in (v,v+\varepsilon)} \|(\nabla^{\frac{H}{2},\frac{H}{2}}Z_s)(u,u)-(\nabla^{\frac{H}{2},\frac{H}{2}}Z_s)(v,v)\|_{L^8(\Omega)} = 0.
	\end{equation*}
	
\medskip
	
\textit{Claim 5:} The process $Z$ satisfies condition \ref{ass:gen_7} of \autoref{prop:Ito_formula}.

\textit{Proof of Claim 5:} As above, by using \autoref{prop:1der_2int} and \autoref{prop:boundedness_of_int_R}, the estimate
	\begin{align*}
		 \esssup_{s,u\in [0,T]}\|(\nabla^\frac{H}{2}Z_s)(u)\|_{\mathbb{D}^{1,8}} & \lesssim \esssup_{s,t\in [0,T]}\left[ \left(\int_0^s\|(\nabla^{\frac{H}{2}}\psi_r)(u)\|_{\mathbb{D}^{3,8}}^\frac{1}{H}\,\mathrm{d}{r}\right)^H + \left(\int_0^s\|\psi_r\|_{\mathbb{D}^{2,8}}^\frac{2}{H+1}|u-r|^{\frac{2(H-1)}{H+1}}\,\mathrm{d}{r}\right)^\frac{H+1}{2}\right]\\
		& \leq \|p_1\|_{L^\frac{1}{H}(0,T)} + \|\psi\|_{L^\infty(0,T;\mathbb{D}^{2,8})}\esssup_{u\in [0,T]}\left(\int_0^s|u-r|^\frac{2(H-1)}{H+1}\,\mathrm{d}{r}\right)^\frac{H+1}{2}.
	\end{align*}
is obtained. The essential supremum in the expression above is clearly finite and the norms are finite by assumptions \ref{ass:psi_1} and \ref{ass:psi_2}. The finiteness of
	\begin{equation*}
		\esssup_{s,u\in [0,T]} \|(\nabla^{\frac{H}{2},\frac{H}{2}}Z_s)(u,u)\|_{L^8(\Omega)}
	\end{equation*}
can be shown similarly.
\end{proof}

\begin{example}
It is now shown that the It\^o-type formula \eqref{eq:Ito_for_integral} is satisfied for any function $f$ in $\mathscr{C}^3(\mathbb{R})$ with bounded third derivative and the process $(Z_W(t))_{t\in [0,T]}$ given by
	\begin{equation*}	
	Z_W(t)\overset{\mathrm{Def.}}{=}\int_0^tW_s\delta R_s^H
	\end{equation*}
where $W$ is the underlying Wiener process. It is necessary to show that the Wiener process $W$ satisfies the conditions on $\psi$ in \autoref{prop:Ito_for_integral_only} with $\alpha=0$. For every $r\in [0,T]$, $W_s = I(\bm{1}_{[0,s]})$. Thus
	\begin{equation*}
		\|W_r\|_{\mathbb{D}^{4,p}} = \left(\mathbb{E}|W_s|^p+\|\bm{1}_{[0,s]}\|_{L^2(\mathbb{R})}^p\right)^\frac{1}{p} \lesssim \sqrt{s}
	\end{equation*}
is satisfied for every $p\geq 2$ by the fact that the process $W$ has equivalent moments. This implies that \ref{ass:psi_1} is satisfied. From \autoref{ex:nabla_W}, it further follows that
	\begin{equation*}
		\nabla^{\frac{H}{2}}W_s(u) \eqsim \left[u_+^\frac{H}{2}-(u-s)_+^\frac{H}{2}\right]
	\end{equation*}
is satisfied for every $u\in\mathbb{R}$ and $\nabla^{\frac{H}{2},\frac{H}{2}}W_s(u,v)=0$ for every $u,v\in\mathbb{R}$. Therefore, \ref{ass:psi_4} and \ref{ass:psi_2} are satisfied as well.
\end{example}

\begin{example}
In this example, it is shown that the It\^o-type formula \eqref{eq:Ito_for_integral} is satisfied for any function $f$ in $\mathscr{C}^3(\mathbb{R})$ with bounded third derivative and the process $(Z_R(t))_{t\in[0,T]}$ given by
	\begin{equation*}
		Z_R(t) \overset{\mathrm{Def.}}{=} \int_0^tR_s^H\delta R_s^H.
	\end{equation*}
Recall from \autoref{def:Rosenblatt_process} that for every $s\in[0,T]$, the Rosenblatt process is given by $R_s^H = C_H^RI_2(h_s^H)$ where
	\begin{equation*}
		h_s^H(x,y) = \int_0^s(u-x)_+^{\frac{H}{2}-1}(u-y)_+^{\frac{H}{2}-1}\,\mathrm{d}{u}.
	\end{equation*}
Since the Rosenblatt process is defined as the second-order multiple Wiener-It\^o integral, all of its moments can be estimated by its second moment, see \cite[Corollary 2.8.14]{NouPec12}. By using this property and formula \eqref{eq:integral_beta}, it follows that
	\begin{equation*}
		\|R_s^H\|_{\mathbb{D}^{4,p}} \lesssim \|h^H_s\|_{L^2(\mathbb{R}^2)} \eqsim s^{H}
	\end{equation*}
is satisfied for every $p\geq 2$. This implies that \ref{ass:psi_1} is satisfied with $\alpha=0$. Using \autoref{prop:1der_2int} and \autoref{prop:2der_2int}, it follows that
	\begin{align*}
		& \nabla^{\frac{H}{2}}R_s^H(u) \eqsim \int_0^s|u-r|^{H-1}\delta B_{r}^{\frac{H}{2}+\frac{1}{2}}\\
		& \nabla^{\frac{H}{2},\frac{H}{2}}R_s^H(u,v) \eqsim \int_0^s|u-r|^{H-1}|v-r|^{H-1}\,\mathrm{d}{r}
	\end{align*}
from which, appealing to \autoref{prop:boundedness_of_int_B}, it follows that both conditions \ref{ass:psi_4} and \ref{ass:psi_2} are satisfied as well.
\end{example}

The following corollary of \autoref{prop:Ito_for_integral_only} provides an It\^o-type formula for functionals of Wiener integrals with respect to the Rosenblatt process, i.e. when the integrand is deterministic.

\begin{corollary}
\label{cor:Ito_for_Wiener_integral}
Let $f\in\mathscr{C}^3(\mathbb{R})$ be such that its third derivative has at most polynomial growth and let $(\psi_t)_{t\in [0,T]}$ be a bounded deterministic function. Let $(Z_t)_{t\in [0,T]}$ be the integral process defined by \eqref{eq:Z}. Then the formula
	\begin{align*}
		f(Z_t) & = f(0) + \int_0^tf'(Z_s)\psi_s\delta R_s^H \\
		&\hspace{5mm} + H(2H-1)\int_0^tf''(Z_s)\psi_s\int_0^s\psi_r(s-r)^{2H-2}\,\mathrm{d}{r}\,\mathrm{d}{s}\\
		&\hspace{1.5cm} + c_1(H)\int_0^tf''(Z_s)\psi_s\left(\int_0^s\psi_r(s-r)^{H-1}\delta B_r^{\frac{H}{2}+\frac{1}{2}}\right)\delta B_s^{\frac{H}{2}+\frac{1}{2}}\\
		& \hspace{5mm} + (\sqrt{2H(2H-1)})^3\int_0^tf'''(Z_s)\psi_s\int_0^s\psi_u(s-u)^{H-1}\int_0^u\psi_v(s-v)^{H-1}(u-v)^{H-1}\,\mathrm{d}{v}\,\mathrm{d}{u}\,\mathrm{d}{s}\\
		& \hspace{1.5cm} + c_2(H) \int_0^t f'''(Z_s)\psi_s\left(\int_0^s\psi_u(s-u)^{H-1}\left(\int_0^u\psi_v(s-v)^{H-1}\delta B_v^{\frac{H}{2}+\frac{1}{2}}\right)\delta B_u^{\frac{H}{2}+\frac{1}{2}}\right)\,\mathrm{d}{s}
	\end{align*}
is satisfied almost surely for every $t\in [0,T]$ with the constants $c_1(H)$ and $c_2(H)$ given by
\begin{equation}
	\label{eq:constants_c_12(H)}
		c_1(H) \overset{\textnormal{Def.}}{=} \frac{4(2H-1)}{H+1}, \qquad c_2(H)\overset{\textnormal{Def.}}{=} \frac{8(2H-1)}{H+1}\sqrt{\frac{H(2H-1)}{2}}.
	\end{equation}
\end{corollary}

\begin{proof}
By \autoref{prop:Ito_for_integral_only}, it follows that
	\begin{align*}
		f(Z_t) & =f(0) + \int_0^t f'(Z_s)\psi_s\delta R_s^H \\
		& \hspace{1cm} + 2c_H^{B,R}\int_0^tf''(Z_s)(\nabla^\frac{H}{2}Z_s)(s)\psi_s\delta B_s^{\frac{H}{2}+\frac{1}{2}} \\
		& \hspace{2cm} + c_H^R\int_0^t \left(f''(Z_s)(\nabla^{\frac{H}{2},\frac{H}{2}}Z_s)(s,s) + f'''(Z_s)[(\nabla^\frac{H}{2}Z_s)(s)]^2\right)\psi_s\,\mathrm{d}{s}.\numberthis\label{eq:not_refined_formula}
	\end{align*}
For the verification of this corollary, it is necessary to compute $\nabla^\frac{H}{2}Z_s(s)$, $[\nabla^\frac{H}{2}Z_s(s)]^2$, and $\nabla^{\frac{H}{2},\frac{H}{2}}Z_s(s,s)$. By \autoref{prop:1der_2int}, it follows that
	\begin{equation*}
		\nabla^{\frac{H}{2}}Z_s(s) = \frac{2c_H^{B,R}\mathrm{B}\left(\frac{H}{2},1-H\right)}{\Gamma\left(\frac{H}{2}\right)^2}\int_0^s\psi_r(s-r)^{H-1}\delta B_r^{\frac{H}{2}+\frac{1}{2}}
	\end{equation*}
is satisfied for almost every $s\in[0,T]$. Since $\nabla^\frac{H}{2}Z_s(s)$ is a Wiener integral with respect to $B^{\frac{H}{2}+\frac{1}{2}}$, its square can be computed by using the chain rule for functionals of Wiener integrals with respect to the fractional Brownian motion, see \cite[Corollary 4.4]{DunHuDun00}. In particular, it follows that the equality
	\begin{align*}
		[\nabla^{\frac{H}{2}}Z_s(s)]^2 & = \frac{8(c_H^{B,R})^2\mathrm{B}\left(\frac{H}{2},1-H\right)^2}{\Gamma\left(\frac{H}{2}\right)^4}\int_0^s\psi_u(s-u)^{H-1}\left(\int_0^u\psi_v(s-v)^{H-1}\delta B_v^{\frac{H}{2}+\frac{1}{2}}\right)\delta B_u^{\frac{H}{2}+\frac{1}{2}}\\
		&\hspace{1mm} + \frac{4H(H+1)(c_H^{B,R})^2\mathrm{B}\left(\frac{H}{2},1-H\right)^2}{\Gamma\left(\frac{H}{2}\right)^4}\int_0^s\psi_u(s-u)^{H-1}\int_0^u\psi_v(s-v)^{H-1}(u-v)^{H-1}\,\mathrm{d}{v}\,\mathrm{d}{u}
	\end{align*}
is satisfied for almost every $s\in [0,T]$. Note that, alternatively, \autoref{prop:Ito_formula} can be used to compute the square. Finally, by \autoref{prop:2der_2int}, it follows that
	\begin{equation*}
		\nabla^{\frac{H}{2},\frac{H}{2}}Z_s(s,s) = \frac{2c_H^R\mathrm{B}\left(\frac{H}{2},1-H\right)}{\Gamma\left(\frac{H}{2}\right)^2}\int_0^s\psi_r(s-r)^{2H-2}\,\mathrm{d}{r}
	\end{equation*}
is satisfied for almost every $s\in [0,T]$. The claim of the corollary is thus verified by substituting the three expressions above in \eqref{eq:not_refined_formula} and simplifying the constants (see \autoref{rem:constants} for their values).
\end{proof}

For the sake of completeness, it is noted that from \autoref{cor:Ito_for_Wiener_integral},  the It\^o formula for functionals of the Rosenblatt process itself can be obtained. This formula is given in \cite[Theorem 3]{Arr16} and proved in the white-noise setting for the case when $f$ is an infinitely differentiable function whose derivatives have at most polynomial growth. Here, it is only required that $f$  be $\mathscr{C}^3$.

\begin{corollary}
\label{prop:Ito_R}
Let $f\in\mathscr{C}^3(\mathbb{R})$ be such that its third derivative has at most polynomial growth. Then the equality
	\begin{align*}
		f(R_t^H) & = f(0) + \int_0^tf'(R_s^H)\delta R_s^H \\
		&\hspace{1cm} + H\int_0^tf''(R_s^H)s^{2H-1} \,\mathrm{d}{s} + \\
		&\hspace{2cm} + c_1(H)\int_0^tf''(R_s^H)\left(\int_0^s(s-u)^{H-1}\delta B_u^{\frac{H}{2}+\frac{1}{2}}\right)\delta B_s^{\frac{H}{2}+\frac{1}{2}}\\
		&\hspace{1cm} +\frac{H}{2}\kappa_3(R_1^H)\int_0^tf'''(R_s^H)s^{3H-1}\,\mathrm{d}{s} \\
		&\hspace{2cm} +c_2(H)\int_0^tf'''(R_s^H)\left(\int_0^s(s-u)^{H-1}\left(\int_0^u(s-v)^{H-1}\delta B_v^{\frac{H}{2}+\frac{1}{2}}\right)\delta B_u^{\frac{H}{2}+\frac{1}{2}}\right)\,\mathrm{d}{s}
	\end{align*}
is satisfied for $t\in [0,T]$ almost surely with the constants $c_1(H)$ and $c_2(H)$ given by \eqref{eq:constants_c_12(H)} and with $\kappa_3(R_1^H)$ being the third cumulant of $R_1^H$ given by
	\begin{equation*}
		\kappa_3(R_1^H) = \frac{4\sqrt{2H(2H-1)^3}}{3H-1}\mathrm{B}(H,H).
	\end{equation*}
where $\mathrm{B}$ is the Beta function.
\end{corollary}

\begin{proof}
Set $\psi\equiv 1$ in \autoref{cor:Ito_for_Wiener_integral}. The cumulant $\kappa_3(R_1^H)$ is computed in \cite[formula (12)]{TaqVei13}.
\end{proof}

\begin{example} As an example, the square of $R^H$ is computed. The following equality is obtained:
	\begin{equation*}
		(R_t^{H})^2 = 2\int_0^tR_s^H\delta R_s^H + \frac{8(2H-1)}{H+1} \int_0^t\left(\int_0^s(s-u)^{H-1}\delta B_{u}^{\frac{H}{2}+\frac{1}{2}}\right)\delta B_{s}^{\frac{H}{2}+\frac{1}{2}} + t^{2H}.
	\end{equation*}
The formula has the same structure as given in \cite[Theorem 4]{Tud08} for the case when $R^H$ is defined by its finite time interval representation from \cite[Proposition 1]{Tud08}. See also \cite[Theorem 3.12]{Arr15} where the square is computed in the white-noise setting. It is interesting to note that the above formula can be written as
	\begin{align*}
		(R_t^H)^2 & = 2\int_0^t\left(\int_0^s\delta R_u^H\right)\delta R_s^H \\
		& \hspace{1cm} + \frac{8(2H-1)}{H+1}\int_0^t\left(\int_0^s(s-u)^{H-1}\delta B_u^{\frac{H}{2}+\frac{H}{2}}\right)\delta B_s^{\frac{H}{2}+\frac{1}{2}}\\
		& \hspace{2cm} + 2H(2H-1)\int_0^t\left(\int_0^s(s-u)^{2H-2}\,\mathrm{d}{u}\right)\,\mathrm{d}{s}.
	\end{align*}
\end{example}

\section{An application}
\label{sec:corollaries}

The It\^o-type formula from \autoref{prop:Ito_for_integral_only} can be used to compute the second moment of the stochastic integral with respect to a Rosenblatt process.

\begin{proposition}
\label{prop:L2norm_of_int}
Let $\psi$ be a stochastic process that satisfies \ref{ass:psi_1} - \ref{ass:psi_2} and $(Z_t)_{t\in[0,T]}$ be defined by \eqref{eq:Z}. Then we have that
	\begin{align*}
		\mathbb{E} \left(Z_t\right)^2 & = H(2H-1)\int_0^t\int_0^t \mathbb{E}[\psi_r\psi_s]|s-r|^{2H-2}\,\mathrm{d}{r}\,\mathrm{d}{s} \\
			& \hspace{1cm} + 2H(2H-1)c_3(H) \int_0^t\int_0^t\mathbb{E}\left[\nabla^{\frac{H}{2}}\psi_r(s) \nabla^\frac{H}{2}\psi_s(r)\right]|s-r|^{H-1}\,\mathrm{d}{r}\,\mathrm{d}{s} \\& \hspace{1cm} + \frac{1}{2}H(2H-1)c_3(H)^2\int_0^t\int_0^t\mathbb{E}\left[\nabla^{\frac{H}{2},\frac{H}{2}}\psi_r(s,s) \nabla^{\frac{H}{2},\frac{H}{2}}\psi_s(r,r)\right]\,\mathrm{d}{r}\,\mathrm{d}{s}\numberthis\label{eq:second_moment}
	\end{align*}
is satisfied with the constant $c_3(H)$ given by
	\begin{equation*}
		c_3(H)\overset{\textnormal{Def.}}{=}\frac{\Gamma\left(\frac{H}{2}\right)\Gamma\left(1-\frac{H}{2}\right)}{\Gamma\left(1-H\right)}.
	\end{equation*}
\end{proposition}

\begin{proof}
Using \autoref{prop:Ito_for_integral_only} with $f(x) = x^2$, it follows that
	\begin{equation*}
		\mathbb{E}Z_t^2 = 2c_H^R\mathbb{E}\int_0^t\left[(\nabla^{\frac{H}{2},\frac{H}{2}}Z_s)(s)\psi_s\right]\,\mathrm{d}{s}
	\end{equation*}
because the stochastic integrals have zero expectation. Using \autoref{prop:2der_2int}, it follows that
	\begin{align*}
		\mathbb{E}\left[ (\nabla^{\frac{H}{2},\frac{H}{2}}Z_s)(s)\psi_s\right] & = \mathbb{E}\left[\psi_s\left(\int_0^s(\nabla^{\frac{H}{2},\frac{H}{2}}\psi_r)(s)\delta R_r^H\right)\right] \\
		& \hspace{1cm} + \frac{4c_H^{B,R}\mathrm{B}\left(\frac{H}{2},1-H\right)}{\Gamma\left(\frac{H}{2}\right)^2} \mathbb{E}\left[\psi_s\left(\int_0^s(\nabla^\frac{H}{2}\psi_r)(s)(s-r)^{H-1}\delta B_r^{\frac{H}{2}+\frac{1}{2}}\right)\right] \\
		& \hspace{1cm} + \frac{2c_H^R\mathrm{B}\left(\frac{H}{2},1-H\right)^2}{\Gamma\left(\frac{H}{2}\right)^4} \int_0^s \mathbb{E}[ \psi_r\psi_s](s-r)^{2H-2}\,\mathrm{d}{r}.
	\end{align*}
Using \ref{lem:adjoint_property_R} of \autoref{lem:adjoint_property} yields
	\begin{equation*}
		\mathbb{E}\left[\psi_s\left(\int_0^s(\nabla^{\frac{H}{2},\frac{H}{2}}\psi_r)(s,s)\delta R_r^H\right)\right] = c_H^R\int_0^s \mathbb{E}\left[(\nabla^{\frac{H}{2},\frac{H}{2}}\psi_s)(r,r)(\nabla^{\frac{H}{2},\frac{H}{2}}\psi_r)(s,s)\right]\,\mathrm{d}{r}
	\end{equation*}
and by \ref{lem:adjoint_property_B} of the same lemma, it follows that
	\begin{equation*}
		\mathbb{E}\left[\psi_s\left(\int_0^s(\nabla^\frac{H}{2}\psi_r)(s)(s-r)^{H-1}\delta B_r^{\frac{H}{2}+\frac{1}{2}}\right)\right] = c_{\frac{H}{2}+\frac{1}{2}}^B\int_0^s \mathbb{E}\left[(\nabla^\frac{H}{2}\psi_s)(r) (\nabla^\frac{H}{2}\psi_r)(s)\right](s-r)^{H-1}\,\mathrm{d}{r}.
	\end{equation*}
is satisfied for almost every $s\in [0,t]$.
\end{proof}

\begin{remark}
A few remarks are made now.
	\begin{enumerate}[label=(\roman*)]
		\item Formula \eqref{eq:second_moment} holds under weaker assumptions: it is sufficient if the integrand $\psi$ belongs to the space $L^\frac{1}{H}(0,T;\mathbb{D}^{4,2})$. This follows because the duality formula from \autoref{lem:adjoint_property} can be used instead of the It\^o formula in which case, the assumptions \ref{ass:psi_1} - \ref{ass:psi_2} are not needed.
		\item Formula \eqref{eq:second_moment} has the same structure as the one given in \cite[Theorem 5.8]{Arr15} for the white-noise-type integral with respect to Rosenblatt processes. However, in contrast to the proof of \cite[Theorem 5.8]{Arr15}, the proof of \autoref{prop:L2norm_of_int} is relatively straightforward.
		\item If $\psi$ is deterministic, then the well-known expression for the second moment of the Wiener integral with respect to the Rosenblatt process,  e.g. \cite[p. 236]{Tud08}, is obtained:
	\begin{equation*}
		\mathbb{E} Z_t^2 = H(2H-1)\int_0^t\int_0^t\psi_r\psi_s|s-r|^{2H-2}\,\mathrm{d}{r}\,\mathrm{d}{s}.
	\end{equation*}
\end{enumerate}
\end{remark}

An estimate for higher absolute moments of the stochastic integral with respect to the Rosenblatt process is now given.

\begin{proposition}
Let $q\geq 3$ and let $\psi$ be a stochastic process that satisfies \ref{ass:psi_1} - \ref{ass:psi_2} with $\alpha=q-2$. Let $(Z_t)_{t\in[0,T]}$ be the stochastic process defined by \eqref{eq:Z}. Then the estimate
	\begin{equation*}
		\|Z_t\|_{L^q(\Omega)}^3 \leq 3(q-1)c_H^R \int_0^t \left\|\psi_s\left(|Z_s|(\nabla^{\frac{H}{2},\frac{H}{2}}Z_s)(s,s) + \mathrm{sgn}(Z_s)(q-2)[(\nabla^\frac{H}{2}Z_s)(s)]^2\right)\right\|_{L^\frac{q}{3}(\Omega)}\,\mathrm{d}{s}.
	\end{equation*}
is satisfied for every $t\in [0,T]$.
\end{proposition}

\begin{proof}
Initially, assume that $q>3$. Using \autoref{prop:Ito_for_integral_only} with $f(x)=|x|^q$ ($f$ is $\mathscr{C}^3$ since $q>3$), it follows that
	\begin{align*}
		|Z_t|^q & = \int_0^tq|Z_s|^{q-1}\mathrm{sgn}(Z_s)\psi_s\delta R_s^H \\
		& \hspace{10mm} + 2c_H^{B,R}\int_0^tq(q-1)|Z_s|^{q-2}\nabla^\frac{H}{2}Z_s(s)\psi_s\delta B_s^{\frac{H}{2}+\frac{1}{2}} \\
		& \hspace{10mm} + c_H^R\int_0^t\psi_s\left(q(q-1)|Z_s|^{q-2}(\nabla^{\frac{H}{2},\frac{H}{2}}Z_s)(s,s) + \right. \\
		& \hspace{4cm}  \left.\phantom{\nabla^\frac{H}{2}} + q(q-1)(q-2)|Z_s|^{q-3}\mathrm{sgn}(Z_s)[(\nabla^\frac{H}{2}Z_s)(s)]^2\right)\,\mathrm{d}{s}\numberthis\label{eq:power}
	\end{align*}
where $\mathrm{sgn}$ denotes the sign function. Taking the expectation of both sides of \eqref{eq:power}, it follows that
	\begin{equation*}
		\mathbb{E}|Z_t|^q = q(q-1)c_H^R\int_0^t\mathbb{E} \left[|Z_s|^{q-3}\psi_s\left(|Z_s|(\nabla^{\frac{H}{2},\frac{H}{2}}Z_s)(s,s) + (q-2)\mathrm{sgn}(Z_s)[(\nabla^\frac{H}{2}Z_s)(s)]^2\right)\right]\,\mathrm{d}{s}
	\end{equation*}
because the stochastic integrals have zero expectation. Thus, H\"older's inequality yields
	\begin{align*}
		\mathbb{E} |Z_t|^q & \leq q(q-1)c_H^R\int_0^t (\mathbb{E} |Z_s|^q)^\frac{q-3}{q} \cdot \\
			&\hspace{2cm} \cdot \left\|\psi_s\left(|Z_s|(\nabla^{\frac{H}{2},\frac{H}{2}}Z_s)(s,s)+(q-2)\mathrm{sgn}(Z_s)[(\nabla^\frac{H}{2}Z_s)(s)]^2\right)\right\|_{L^\frac{q}{3}(\Omega)}\,\mathrm{d}{s}.
	\end{align*}
The desired inequality is proved by using Bihari's inequality, see \cite[Theorem 3, p. 135]{BecBel61}, which gives
	\begin{equation*}
		\mathbb{E} |Z_t|^q \leq \left(\frac{3}{q}q(q-1)c_H^R\int_0^t\left\|\psi_s\left(|Z_s|(\nabla^{\frac{H}{2},\frac{H}{2}}Z_s)(s,s)+(q-2)\mathrm{sgn}(Z_s)[(\nabla^\frac{H}{2}Z_s)(s)]^2\right)\right\|_{L^\frac{q}{3}(\Omega)}\,\mathrm{d}{s}\right)^\frac{q}{3}.
	\end{equation*}
For the case $q=3$, \autoref{prop:Ito_for_integral_only} cannot be used directly, since the function $f(x)=|x|^3$ does not belong to $\mathscr{C}^3(\mathbb{R})$. Instead, for $\varepsilon>0$, consider the function
	\begin{equation*}
		f_\varepsilon(x) \overset{\textnormal{Def.}}{=} (x^2+\varepsilon^2)^\frac{3}{2}, \quad x\in\mathbb{R}.
	\end{equation*}
The function $f_\varepsilon$ is a smooth approximation of $f(x)=|x|^3$ with bounded third derivative. Hence, by \autoref{prop:Ito_for_integral_only} it follows that $\mathbb{E} f_\varepsilon(Z_t)$ satisfies the formula
	\begin{equation}
	\label{eq:approximate_eq}
		\mathbb{E} f_\varepsilon(Z_t) = \varepsilon^3 + c_H^R\int_0^t \mathbb{E}\left[\psi_s\left(f''_\varepsilon(Z_s)(\nabla^{\frac{H}{2},\frac{H}{2}}Z_s)(s,s) + f'''_\varepsilon(Z_s)[(\nabla^\frac{H}{2}Z_s)(s)]^2\right)\right]\,\mathrm{d}{s}
	\end{equation}
similarly as in the case $q>3$. Since
	\begin{equation*}
		\lim_{\varepsilon\downarrow 0}f''_\varepsilon(x)=6|x|\quad\mbox{and}\quad\lim_{\varepsilon\downarrow 0}f'''_\varepsilon(x) = 6\mathrm{sgn}(x),
	\end{equation*}
taking the limit $\varepsilon\downarrow 0$ in equality \eqref{eq:approximate_eq} and using Lebesgue's dominated convergence theorem to interchange the limit and the integrals yields
	\begin{equation*}
		\mathbb{E} |Z_t|^3 = 6c_H^R \int_0^t \mathbb{E}\left[\psi_s\left(|Z_s|(\nabla^{\frac{H}{2},\frac{H}{2}}Z_s)(s,s) + \mathrm{sgn}(Z_s)[(\nabla^\frac{H}{2}Z_s)(s)]^2\right)\right]\,\mathrm{d}{s}
	\end{equation*}
which concludes the proof.
\end{proof}

\section{A concluding remark}
Stochastic processes with second-order fractional stochastic differential arise naturally in the It\^o-type formula for functionals of (stochastic integral with respect to) Rosenblatt processes. Such formulas will in general contain stochastic integrals with respect to a fractional Brownian motion as well as derivatives up to the third order.

It seems that the method given here can be also used to obtain It\^o-type formulas for higher-order processes, e.g. Hermite processes, and it appears that such formulas will contain stochastic integrals with respect to related fractional processes of lower Hermite order as well as derivatives up to the Hermite order of the considered process increased by one.

Another topic which should be further studied is non-linear stochastic differential equations with Rosenblatt noise. The general It\^o-type formula given here should provide a convenient tool for their analysis. However, important questions of existence of the solutions need to be answered and further properties of the solutions and of fractional stochastic derivatives of the solutions need to be investigated before the formula can be applied.

\section*{Acknowledgement}
The first author is grateful to Jana \v{S}nup\'arkov\'a for her careful reading of the first version of this paper and to Martin Ondrej\'at and Bohdan Maslowski for their helpful suggestions.

This research was supported by NSF grant DMS 1411412 and AFOSR grant FA9550-17-1-0073. The first author was partially supported by the Czech Science Foundation, project GA\v{C}R 19-07140S.

\end{document}